\documentclass[a4paper,12pt]{amsproc}

\usepackage{xargs}
\usepackage{mathtools, todonotes}
\usepackage{amsmath,amssymb,amsfonts,amsthm}

\usepackage[colorlinks, backref]{hyperref}
\usepackage{amsrefs}
\usepackage{enumerate}

\AddToHook{bfseries}{\boldmath}

\usepackage[
  a4paper,
  left=27mm,
  right=27mm,
  top=27mm,
  bottom=30mm
]{geometry}

\newcommand{\norm}[1]{\left\lvert#1\right\rvert}
\newcommand{\card}[1]{\left\lvert#1\right\rvert}
\newcommand{\perpnotin}{\mathrel{\perp\!\!\!\perp}}
\DeclareMathOperator{\diam}{diam}
\DeclareMathOperator{\Sym}{Sym}
\DeclareMathOperator{\GL}{GL}
\DeclareMathOperator{\rk}{rk}
\DeclareMathOperator{\GI}{GI}
\DeclareMathOperator{\GO}{GO}
\DeclareMathOperator{\SO}{SO}
\DeclareMathOperator{\Sp}{Sp}
\DeclareMathOperator{\SU}{SU}
\DeclareMathOperator{\SL}{SL}
\DeclareMathOperator{\GU}{GU}
\DeclareMathOperator{\PSU}{PSU}
\DeclareMathOperator{\PSL}{PSL}
\DeclareMathOperator{\PSp}{PSp}
\DeclareMathOperator{\POmega}{P\Omega}
\DeclareMathOperator{\PGO}{PGO}
\DeclareMathOperator{\Sz}{Sz}
\DeclareMathOperator{\Alt}{Alt}
\newcommandx{\gensubgrp}[2][2=\empty]{\langle#1\ifx#2\empty\else\,|\,#2\fi\rangle}
\newcommandx{\gennorsubgrp}[2][2=\empty]{\langle\!\langle#1\ifx#2\empty\else\,|\,#2\fi\rangle\!\rangle}
\newcommandx{\gensubsp}[2][2=\empty]{\langle#1\ifx#2\empty\else\,|\,#2\fi\rangle}
\newcommand{\rest}[1]{\left.#1\right\rvert}
\newcommandx{\set}[2][2=\empty]{\{#1\ifx#2\empty\else\,|\,#2\fi\}}
\DeclareMathOperator{\diag}{diag}
\DeclareMathOperator{\ad}{ad}
\DeclareMathOperator{\im}{im}
\DeclareMathOperator{\Aut}{Aut}
\DeclareMathOperator{\End}{End}
\DeclareMathOperator{\lcm}{lcm}
\DeclareMathOperator{\adj}{adj}
\DeclareMathOperator{\fix}{fix}
\DeclareMathOperator{\FGL}{FGL}
\DeclareMathOperator{\FSp}{FSp}
\DeclareMathOperator{\FOmega}{F\Omega}
\DeclareMathOperator{\FSU}{FSU}
\DeclareMathOperator{\PFGL}{PFGL}

\DeclareMathOperator{\id}{id}

\theoremstyle{plain}
\newtheorem{definition}{Definition}[section]
\newtheorem{theorem}[definition]{Theorem}
\newtheorem{lemma}[definition]{Lemma}
\newtheorem{corollary}[definition]{Corollary}
\newtheorem{proposition}[definition]{Proposition}

\newtheorem{thmintro}{Theorem}

\newtheorem{conjintro}[thmintro]{Conjecture}
\newtheorem{conjsintro}[thmintro]{Conjectures}

\theoremstyle{definition}
\newtheorem{remark}[definition]{Remark}
\newtheorem{example}[definition]{Example}

\title{Mixed identities for simple locally finite groups}
\author[H. Bradford]{Henry Bradford}
\address{H.~Bradford, Christ's College, Univ.\ of Cambridge, St Andrew's St., Cambridge CB2 3BU, United Kingdom}
\email{hb470@cam.ac.uk}
\author[K. Ersoy]{Kıvanç Ersoy}
\address{K.~Ersoy, Freie Universität Berlin, 14195 Berlin, Germany}
\email{ersoy@zedat.fu-berlin.de}
\author[J. Schneider]{Jakob Schneider}
\address{J.~Schneider, TU Dresden, 01062 Dresden, Germany}
\email{jakob.schneider@tu-dresden.de}
\author[A. Thom]{Andreas Thom}
\address{A.~Thom, TU Dresden, 01062 Dresden, Germany}
\email{andreas.thom@tu-dresden.de}

\begin{document}
\begin{abstract}
A mixed identity of a group is a nontrivial word with constants that
vanishes under every substitution of its variables. 
We derive lower bounds for the length of mixed identities in finite simple groups of Lie type, and characterise exactly those families of such groups of bounded rank 
which satisfy mixed identities of bounded length. We classify the infinite simple locally finite groups admitting a mixed identity: apart from an explicit list of alternating, finitary linear classical, and non-simply-laced groups of Lie type, no such group exists. For the groups in this list, we determine when mixed identities must be singular and obtain restrictions on their critical constants. Moreover, we prove that simple compact Lie groups do not admit mixed identities.
\end{abstract}
\maketitle

\tableofcontents
	
\section{Introduction}

\emph{Group identities} (also called \emph{group laws}) are a classical tool for detecting algebraic constraints in groups. Recall that a law for the group $G$ is a nontrivial element $w$ 
of the free group $\mathbf{F}_r$ whose associated word map $G^r\rightarrow G$ 
vanishes on $G^r$. A theorem of Jones~\cite{jones1974varieties} shows that a fixed nontrivial identity is
satisfied by only finitely many isomorphism classes of nonabelian finite
simple groups. Allowing the identity also to contain constants from the group leads to the more flexible
notion of a \emph{mixed identity}. A special class of mixed identities are those that collapse completely after removal of the constants, they are called \emph{singular}. A particular class of nonsingular mixed identities is formed by those for which no cancellation of variables occurs after removal of all constants, they are called \emph{strict}.
In contrast to ordinary identities, mixed
identities reflect not only the global structure of the group but also the
position and conjugacy properties of the constants involved.

The systematic study of mixed identities goes back to Anashin
\cites{anashin1978mixed,anashin1987mixed}. Foundational results for
classical, linear, and algebraic groups were obtained by
Golubchik--Mikhalev \cite{golubchikmikhalev1982generalized}, Tomanov \cite{tomanov1985generalized}, and Gordeev \cite{gordeev1997freedom}. In particular, Tomanov proved that a linear group
admits a strict mixed identity if and only if it is solvable-by-finite. More recently, Hull and Osin related the
absence of mixed identities to highly transitive actions and
acylindrical hyperbolicity \cite{hullosin2016transitivity}, while
Gordeev, Kunyavski\u{\i}, and Plotkin studied word maps with constants on
simple algebraic groups \cite{gordeevkunyavskiiplotkin2018word}. The asymptotic behaviour of the lengths of mixed identities satisfied by 
finite groups has been studied by subsets of the authors 
\cites{bradfordschneiderthom2023length,bradfordschneiderthom2023non,schneiderthom2022word}.

Infinite simple locally finite groups provide a natural setting for this
study. They are assembled from finite simple groups, but their global
structure ranges from linear groups over locally finite fields to genuinely
nonlinear direct limits. 
For general background on locally finite groups and surveys of the
simple case, see
\cites{kegelwehrfritz1973locally,hartley1995simple,meierfrankenfeld2011locally}.

The purpose of this paper is twofold. First, we seek to determine which
infinite simple locally finite groups admit mixed identities and, for those which do so, to describe the restrictions imposed on the constants. 
Second, we seek to bound the length of a shortest mixed identity 
satisfied by a given nonabelian finite simple group from below, 
with an emphasis on groups of Lie type of bounded rank. 
The results we achieve in pursuit of the second goal shall also be 
tools in pursuit of the first.

Let $G$ be a simple locally finite group (which may be finite or infinite). The group $G$ is called \emph{linear} if it occurs as a subgroup of $\GL_n(K)$ for some finite $n$ and a field $K$ and \emph{nonlinear} otherwise. It is a well-known and celebrated result of Belyaev, Borovik, Hartley--Shute, and Thomas \cites{belyaev1984locally, borovik1983embeddings, hartleyshute1984monomorphisms, thomas1983classification} that every infinite linear simple locally finite group is of the form $G=[{}^a\mathbf{G}(E),{}^a\mathbf{G}(E)]$ for an infinite locally finite field $E$ and a simple group of Lie type $\mathbf{G}$ which is twisted by an outer automorphism of order $a\in\set{1,2,3}$, where $a=1$ corresponds to the untwisted case, i.e.
$$
{}^a\mathbf{G}\in\set{A_n,{}^2A_n,B_n,C_n,D_n,{}^2D_n, E_6,{}^2E_6, E_7,E_8, F_4, {}^2F_4,G_2, {}^2G_2, {}^2B_2,  {}^3D_4}. 
$$
Here we need to consider the derived subgroup for $G$ if and only if we are in an orthogonal case $B_n$, $D_n$, ${}^2D_n$, ${}^3D_4$. In all other cases ${}^a\mathbf{G}(E)$ is already perfect.

A larger class of simple locally finite groups consists of those that act by finitary linear transformations on a vector space over a locally finite field. By the classification theorem of  Hall \cite{hall2006periodic}, those include in addition the nonlinear groups
$\Alt(\Omega)$, where $\Omega$ is an infinite set, and the classical groups $\FSp(V,f)$, $\FSU(V,f)$, $\FOmega(V,Q)$, and $\SL(V,W,m)$ over a locally finite field $E$, where $V$ is an infinite-dimensional $E$-vector space with a nondegenerate symplectic resp.\ unitary form $f$, or a quadratic form $Q$ on it, and $m$ is a nondegenerate pairing of the infinite-dimensional
$E$-spaces $V$ and $W$ (cf.\ Theorem~\ref{thm:hall} below).

\medskip

Our first main theorem is as follows.

\begin{thmintro}\label{thm:main}
Let $G$ be an infinite simple locally finite group. Then $G$ admits a mixed identity if and only if it is isomorphic to a group in the following list:
\begin{enumerate}[(i)]
\item $\Alt(\Omega)$ for an infinite set $\Omega$;
\item a classical finitary linear group over a finite field;
\item a finitary symplectic group over an infinite locally finite field;
\item a simple group of Lie type $B_n$, $C_n$, $F_4$, or $G_2$ over an infinite locally finite field.
\end{enumerate}
Moreover, in  Cases~(i)--(iii), every mixed identity is singular.
\end{thmintro}

There are corresponding quantitative statements for nonabelian finite simple groups relating the field size and the rank to the length of a shortest mixed identity. Our second main result is as follows:

\begin{thmintro}\label{thm:main_finintro}
For a fixed possibly twisted Lie type ${}^a\mathbf{G}$, the simple groups $G={}^a\mathbf{G}(q)'$ admit
mixed identities of length bounded independently of $q$ if and only if
$a=1$ and $\mathbf{G}\in\set{B_n,C_n,F_4,G_2}$. For every remaining type 
$$
{}^a\mathbf{G}\in\set{A_n, D_n, E_6, E_7, E_8, 
{}^2A_n, {}^2D_n, {}^3D_4, {}^2E_6, {}^2B_2, {}^2F_4, {}^2G_2}
$$ 
the length of a shortest mixed identity for the group $G={}^a\mathbf{G}(q)'$ is of length $\Omega_{\mathbf{G}}(q^{1/a})$.
\end{thmintro}

In view of Theorems~\ref{thm:main} and \ref{thm:main_finintro}, the following two related conjectures seem natural.

\begin{conjsintro}\label{conj:main_lfsg}
The following hold:
\begin{enumerate}[(i)]
\item Every mixed identity of an infinite simple locally finite group is singular.
\item \cite{bradfordschneiderthom2023non}*{Conjecture~1}
The number of isomorphism classes of nonabelian finite simple groups satisfying a nonsingular mixed identity of fixed length is finite.
\end{enumerate}
\end{conjsintro}

Theorem~\ref{thm:main} establishes Conjecture~\ref{conj:main_lfsg}(i) for all  infinite simple locally finite groups apart from  groups of Lie type $B_n$, $C_n$, $F_4$, or $G_2$ over an infinite locally finite field.
For oligomorphic automorphism groups, Bodirsky and two of the authors formulated the analogous conjecture that every mixed identity is singular \cite{bodirskyschneiderthom2025mixed}. Recent work of Marimon and Pinsker proves this conclusion for every group admitting an action with no algebraicity \cite{marimonpinsker2026all}.

Conjecture~\ref{conj:main_lfsg}(ii) is motivated by the above mentioned result of Jones~\cite{jones1974varieties}. Similarly, Theorem~\ref{thm:main_finintro} proves Conjecture~\ref{conj:main_lfsg}(ii) apart from simple groups of the Lie types $B_n, C_n, F_4$, and $G_2$. Analogous quantitative questions concerning a shortest law in finite
simple groups of Lie type were initiated by Hadad 
\cite{hadad2011shortest} and developed further in 
\cites{kozmathom2016divisibility,thom2017length,bradfordthom2026short}.

In \cite{bradfordschneiderthom2023length} it was proven that any finite group that is 
not almost simple satisfies a mixed identity of length at most eight. 
Therefore, in studying mixed identities for finite groups, attention naturally focuses 
on the simple case.
Note that by the classification of finite simple groups and apart from finitely many sporadic groups, every nonabelian finite simple group is either alternating or of the form ${}^a\mathbf{G}(q)'\coloneqq {}^a\mathbf{G}(\mathbb F_q)'$ with ${}^a\mathbf G$ in the aforementioned list and $\mathbb F_q$ a finite field of size $q$ with some restrictions on $q$. Again, taking the derived subgroup is only relevant in the orthogonal cases. We have a conjectural description of those finite simple groups satisfying 
mixed identities of bounded length.

\begin{conjintro} \label{IntroFSGMainConj}
If $\mathcal{F}$ is a family of pairwise nonisomorphic nonabelian finite simple groups satisfying mixed identities of bounded length, then $\mathcal{F}$ contains a cofinite subfamily 
consisting of groups of the following types: 
\begin{enumerate}[(i)]
\item the alternating groups $\Alt_n$; 
\item groups of Lie type over fields of bounded size; 
\item $C_m(q)=\PSp_{2m}(q)$; 
$B_m(q)'=\POmega_{2m+1}(q)$; $F_4 (q)$; 
$G_2(q)$. 
\end{enumerate}
Conversely, groups of types (i), (ii) or (iii) 
do satisfy mixed identities of bounded length.
\end{conjintro}

Taken together with results from \cite{bradfordschneiderthom2023length} we 
can resolve Conjecture \ref{IntroFSGMainConj} with the possible exception of 
certain families of even-degree orthogonal groups.

\begin{thmintro} \label{thm:OrthCaseImpliesConj}
Conjecture \ref{IntroFSGMainConj} is satisfied by all families $\mathcal{F}$ 
which do not contain infinite sequences of finite simple orthogonal groups 
$\POmega_{2m_i}^\pm(q_i)$ for which $m_i$ and $q_i$ both tend to $\infty$.
\end{thmintro}

In the more general setting of almost simple groups, it is not clear what behaviour to expect. As noted in \cite{bradfordschneiderthom2023length}*{Theorem~3} it can happen that a shortest mixed identity for a finite almost simple group 
can be shorter than a shortest mixed identity for its socle. Here we present an even starker example.

\begin{example}
Let $G_{\pm}\coloneqq{}^\bullet D_m(q)'=\POmega_{2m}^\pm(q)$ and $C_{\pm}\coloneqq\PGO_{2m}^\pm(q)$, so that $C_{\pm}$ is almost simple with socle $G_{\pm}$ of index at most four. We show in Theorem~\ref{thm:main_fin}(i) and (iii) that the shortest mixed for $G_{\pm}$ has length $\Omega_m(q)$, however for $q$ odd there is such an identity of length eight for $C_{\pm}$ (see the Appendix~\ref{sec:almost_simple_grps}). 
\end{example}

In the cases of variable rank, we can also control the length of nonsingular mixed identities:

\begin{thmintro} \label{thm:rank_bound}
Every nonsingular mixed identity of the simple groups $G=\Alt_n$ or $G={}^a \mathbf G(q)'$ for ${}^a \mathbf{G}\in\set{A_n, B_n, C_n, D_n, {}^2A_n, {}^2D_n}$, with $q$ a prime power, is of length $$\Omega\left(\frac{\log(n)}{\log\log(n)} \right).$$
\end{thmintro}

We show that Conjecture~\ref{conj:main_lfsg}(ii) implies Conjecture \ref{conj:main_lfsg}(i); the converse implication is unclear to us. We also deduce the special case of both conjectures for strict mixed identities.

\begin{thmintro} \label{thm:strictsimpleintro}
The following hold:
\begin{enumerate}[(i)]
\item No infinite simple locally finite group has a strict mixed identity.
\item The number of isomorphism classes of nonabelian finite simple groups satisfying a strict mixed identity of fixed length is finite.
\end{enumerate}
\end{thmintro}

\medskip

In various particular cases we prove more refined results not stated in Theorem~\ref{thm:main} and \ref{thm:main_finintro}. Those will be stated in the next section after the necessary terminology has been introduced rigorously. In fact Theorems \ref{thm:main} and \ref{thm:main_finintro} are immediate consequences of the more refined Theorems \ref{thm:main_slfg} and \ref{thm:main_fin} in the next section, and the classification given in Theorem~\ref{thm:class_sing_mxd_id}.

\medskip

The paper is organized as follows: Section~\ref{sec:def} contains definitions and more refined statements of our results (see Theorems~\ref{thm:main_slfg} and \ref{thm:main_fin} 
for the results on infinite locally finite and finite groups, respectively). We establish Theorem~\ref{thm:OrthCaseImpliesConj} and Theorem~\ref{thm:part_conj_jones}, 
which establishes a part of Conjecture~\ref{conj:main_lfsg}(ii). In Section~\ref{sec:classical_grps_of_lie_type} we concentrate on classical groups of Lie type and prove Theorem~\ref{thm:rank_bound} and Theorem~\ref{thm:strictsimpleintro}(ii) (applying Theorem~\ref{thm:main_fin}).  In Section~\ref{sec:alg_grps} we establish Theorem~\ref{thm:main_slfg} and Theorem~\ref{thm:main_fin}. The statements on finite simple groups in Theorem~\ref{thm:main_fin} are proven by tracing back the case of a (twisted) finite simple group of Lie type to the case of an ambient algebraic group over the infinite field $\overline{\mathbb{F}}_q$ using a Schwartz--Zippel bound. Moreover, we remark (see Corollary~\ref{cor:comp_lie_grps}) that, by the results of Gordeev~\cite{gordeev1997freedom}, no centreless simple compact Lie group admits a mixed identity, proving a conjecture of Larsen--Shalev \cite{larsenshalev2022identities}.
In Section~\ref{sec:nonlin_simple_loc_fin_grps} we use \emph{Kegel covers} (see Definition~\ref{def:sec_cov} below) to prove Theorem~\ref{thm:main_slfg}(i) and Theorem~\ref{thm:class_sing_mxd_id}, classifying all nonlinear simple locally finite groups admitting a mixed identity. We also prove that Conjecture~\ref{conj:main_lfsg}(ii) implies Conjecture~\ref{conj:main_lfsg}(i).

\section{Definitions}\label{sec:def}

We now introduce the terminology of mixed identities used to state our results more formally.
Let $G\leq C$ be groups. Let $w\in C\ast\mathbf{F}_r$ be a \emph{word with constants} in $C$, which we call the group of constants. Here $\mathbf{F}_r=\gensubgrp{x_1,\ldots,x_r}$ denotes the \emph{free group} with basis $\set{x_1,\ldots,x_r}$. Subsequently, we fix $w$ to be of the shape
$$
w=c_0x_{i(1)}^{\varepsilon(1)}c_1\cdots c_{l-1}x_{i(l)}^{\varepsilon(l)}c_l,
$$
where $i(j)=i(j+1)$ and $\varepsilon(j)=-\varepsilon(j+1)$ implies that $c_j\neq 1_C$, i.e.\ the representation of $w$ is \emph{reduced}. Here the $c_j$ belong to $C$ and $\varepsilon(j)\in\set{\pm1}$ (for all $j$). We define the \emph{length} of $w$ to be $\norm{w}\coloneqq l$. Moreover, we define the sets 
\begin{align*}
J_0(w)&\coloneqq\set{j\in\set{1,\ldots,l-1}}[i(j)\neq i(j+1)];\\ 
J_+(w)&\coloneqq\set{j\in\set{1,\ldots,l-1}}[i(j)=i(j+1)\text{ and }\varepsilon(j)=\varepsilon(j+1)];\\
J_-(w)&\coloneqq\set{j\in\set{1,\ldots,l-1}}[i(j)=i(j+1)\text{ and }\varepsilon(j)=-\varepsilon(j+1)].
\end{align*}
These will be used a lot in the text. The $c_j$ ($j\in J_-(w)$) are called \emph{critical constants} and we collect them in the set $I(w)\coloneqq\set{c_j}[j\in J_-(w)]$. Let $\Lambda\subseteq C$ be a `small' subset. The word $w$ is called \emph{$\Lambda$-noncentral} (according to \cite{golubchikmikhalev1982generalized}) if every critical constant does not belong to $\Lambda$, i.e.\ $\Lambda\cap I(w)=\varnothing$. So $w$ is an arbitrary (reduced) word when $\Lambda=\mathbf{1}$. When $\Lambda=C$, it follows that $J_-(w)=\varnothing$. In this case, we call $w$ \emph{strict} (according to \cite{tomanov1985generalized}).
We define $\norm{w}_{\rm crit}\coloneqq\min\set{\norm{c}}[c\in I(w)]$ to be the \emph{critical length} of $w$. Here $\norm{\bullet}\colon C\to\mathbb{N}$ denotes a \emph{seminorm} on $C$ which is understood. We set $\norm{w}_{\rm crit}\coloneqq\diam(C)$ for strict words $w$, where $\diam(C)\coloneqq\sup\set{\norm{c}}[c\in C]\in\mathbb{N}\cup\set{\infty}$ is the \emph{diameter} of the group $C$.
	
Write $\varepsilon\colon C\ast\mathbf{F}_r\to\mathbf{F}_r$ for the \emph{augmentation homomorphism} that sends $C\ni c\mapsto 1_{\mathbf{F}_r}$ and fixes each element of $\mathbf{F}_r$. Call 
$$
\varepsilon(w)=x_{i(1)}^{\varepsilon(1)}\cdots x_{i(l)}^{\varepsilon(l)}\in\mathbf{F}_r
$$ 
the \emph{content} of $w$ (here cancellation may occur). The word $w$ is called \emph{singular} if it has trivial content and \emph{nonsingular} otherwise. Moreover, $w$ is called a \emph{mixed identity} for $G$ \emph{with constants} in $C$, if $w\neq 1_{C\ast\mathbf{F}_r}$ and the associated \emph{word map} $w\colon G^r\to C$; $(g_1,\ldots,g_r)\mapsto w(g_1,\ldots,g_r)$, which is defined by substituting the variable $x_i$ for $g_i\in G$ (for all $i\in\set{1,\ldots,r}$), has trivial image, i.e.\ $w(G^r)=\set{1_C}$. When $G=C$, $w$ is just called a mixed identity for $G$.

\medskip
	
The purpose of this article is to prove Theorems \ref{thm:main_slfg} 
and \ref{thm:main_fin} below and thereby, 
to establish parts of Conjectures \ref{conj:main_lfsg}(i) and (ii),
and \ref{IntroFSGMainConj}. 
Theorem~\ref{thm:main_slfg} deals with the case that $G$ is an infinite simple locally finite group, whereas Theorem~\ref{thm:main_fin} focuses on the case when $G$ is a finite simple group. To state both theorems, we need to fix the following exceptional condition:
\begin{equation}\label{eq:exc_cses}
p=2\text{ and } \mathbf{G}=B_n,C_n,F_4;\text{ or }p=3\text{ and }\mathbf{G}=G_2.\tag{$\star$}
\end{equation}
Also note that, when $w\in C\ast\mathbf{F}_r$ is a shortest mixed identity for $G$ (with constants in $C$), there is a mixed identity $w'\in C\ast\gensubgrp{x}$ of the same length by \cite{bradfordschneiderthom2023length}*{Lemma~2.2}. Hence we may restrict to the case that $r=1$ in most cases.
    
\begin{theorem}\label{thm:main_slfg}
Let $G$ be an infinite simple locally finite group.
\begin{enumerate}[(i)]
\item If $G$ is nonlinear, there exists no nonsingular mixed identity for $G$. We classify all cases in which $G$ has a singular mixed identity.
\item When $G=\mathbf{G}(E)'$ is linear (for an infinite locally finite field $E\leq\overline{\mathbb{F}}_q$) and $\mathbf{G}$ is \emph{simply laced}, i.e.\ it is one of $A_n,D_n,E_6,E_7,E_8$, then $G$ has no mixed identities $w\in\mathbf{G}(\overline{\mathbb{F}}_q)\ast\gensubgrp{x}$. 
\item If $\mathbf{G}$ is not simply laced, i.e.\ one of $B_n,C_n,F_4, G_2$, then there is a singular mixed identity for $G=\mathbf{G}(E)'$ with constants in itself; see Theorem~\ref{thm:gordeev_short_ids}. If Condition~\eqref{eq:exc_cses} fails, every mixed identity $w\in\mathbf{G}(\overline{\mathbb{F}}_q)\ast\mathbf{F}_r$ for $G$ has both G-small semisimple and G-small almost unipotent critical constants. If \eqref{eq:exc_cses} holds, then there are G-small almost unipotent critical constants of both classes, see  Definition~\ref{def:gord_smll} and \ref{def:sml_upts} for G-small semisimple and almost unipotent elements.
\item In the twisted case, when $G={}^a\mathbf{G}(E)'$ for
$$
{}^a\mathbf{G}(E)\in\set{{}^2A_n(E),{}^2D_n(E),{}^3D_4(E),{}^2E_6(E)}
$$ 
is a generalized Steinberg group over an infinite locally finite field $E$ with an automorphism $\alpha\in\Aut(E)$ of order $a\in\set{2,3}$, then $G$ also has no mixed identities $w\in\mathbf{G}(\overline{\mathbb{F}}_q)\ast\gensubgrp{x}$.
\item Finally, when $G={}^2\mathbf{G}(E)\in\set{{}^2B_2(E),{}^2F_4(E),{}^2G_2(E)}$ is a generalized Suzuki or Ree group over a suitable infinite locally finite field $E$, then $G$ has no mixed identities (with constants in $G$ itself).
\end{enumerate}     
\end{theorem}

\begin{remark}
The classification of groups $G$ as in (i) having a singular mixed identity is given in Theorem~\ref{thm:class_sing_mxd_id} below. 
\end{remark}

\begin{remark}
Note that, in contrast to (iv), in (v) there are mixed identities with constants in the ambient algebraic group, which is why we considered the two cases separately. Indeed, e.g.\ $\Sz(E)={}^2B_2(E)$ sits inside $\mathbf{G}(\overline{\mathbb{F}}_q)=\Sp_4(\overline{\mathbb{F}}_q)$ which has singular mixed identities by Theorem~\ref{thm:main_slfg}(iii). These carry over to $\Sz(E)$. 
\end{remark}

In the next theorem, we consider finite simple groups of Lie type and prove bounds on a shortest mixed identity for every such group. We will make frequent use of the Landau notation $O$, $\Omega$, and $\Theta$. Recall that $f(n)=\Omega(g(n))$ resp.\ $f(n)=O(g(n))$ if and only if there is $N\in\mathbb{N}$ and a constant $C\in\mathbb{R}_+$ with $f(n)\geq Cg(n)$ resp.\ $f(n)\leq Cg(n)$ for all $n\geq N$.
Moreover, $f(n)=\Theta(g(n))$ if and only if $f(n)=O(g(n))=\Omega(g(n))$.

\begin{theorem}\label{thm:main_fin}
Let $G$ be a finite simple group of Lie type.
\begin{enumerate}[(i)]
\item Let $\mathbf{G}$ be one of the simply laced Lie types $A_n,D_n,E_6,E_7,E_8$.
Then a shortest mixed identity $w\in\mathbf{G}(\overline{\mathbb{F}}_q)\ast\gensubgrp{x}$ for $G=\mathbf{G}(q)'$ is of length $\Omega_{\mathbf{G}}(q)$.
\item If $\mathbf{G}$ is not simply laced, i.e.\ one of $B_n,C_n,F_4, G_2$, 
then there exists a singular mixed identity for $G=\mathbf{G}(q)'$ with constants in $G$ of bounded length. Let $w\in\mathbf{G}(\overline{\mathbb{F}}_q)\ast\mathbf{F}_r$ be a mixed identity for $G$ and assume that $l=\norm{w}=O_{\mathbf{G}}(q)$ for a small enough implicit constant. Then, if \eqref{eq:exc_cses} fails, $w$ contains G-small semisimple and G-small almost unipotent critical constants. When \eqref{eq:exc_cses} holds, there are G-small almost unipotent critical constants of both classes.
\item When $G={}^a\mathbf{G}(q^a)'$ is a Steinberg group ${}^2A_n(q^2),{}^2D_n(q^2),{}^3D_4'(q^3),{}^2E_6(q^2)$, a shortest mixed identity $w\in\mathbf{G}(\overline{\mathbb{F}}_q)\ast\gensubgrp{x}$ for $G$ is of length $\Omega_{\mathbf{G}}(q)$.
\item If $G = {}^2\mathbf{G}(q)$ is a Suzuki group ${}^2B_2(q)$, or a Ree group ${}^2F_4(q)$ or ${}^2G_2(q)$, with $q=2^{2e+1}$ for ${}^2B_2$ and ${}^2F_4$, and $q=3^{2e+1}$ for ${}^2G_2$, then any mixed identity of $G$ (with constants in $G$) is of length $\Omega(q^{1/2})$.
\end{enumerate}
\end{theorem}

From Theorem~\ref{thm:main_fin} we deduce Theorem~\ref{thm:OrthCaseImpliesConj} 
from the Introduction.

\begin{proof}[Proof of Theorem~\ref{thm:OrthCaseImpliesConj}]
First, we show that groups of types (i), (ii) and (iii) do indeed 
admit mixed identities of bounded length. 
For $\Alt_n$, $w(x)=[\sigma,x] ^{30}$ for $\sigma=(1\,2\,3)$ is a mixed identity. 
Case~(iii) is covered by Theorem~\ref{thm:main_fin}(ii). 
For Case~(ii), a family of groups of Lie type of unbounded order over fields of bounded size 
must contain a cofinal subfamily of classical groups. 
This case is covered by Proposition \ref{BddFieldSizeProp} below.

Conversely, suppose our claim fails and let $(G_i)_{i\in\mathbb{N}}$ be 
a witnessing sequence of pairwise nonisomorphic 
nonabelian finite simple groups satisfying mixed identities of bounded length which, 
excluding sporadic groups and the groups appearing in Cases~(i) and (ii), 
we may assume consists of groups of Lie type over fields of unbounded size. 
By \cite{bradfordschneiderthom2023length}*{Theorems~3 and 4} we may exclude 
projective special linear and unitary groups. 
By Theorem~\ref{thm:main_fin} we may exclude symplectic groups ($\Sp_{2m}(q)=C_m(q)$), odd-degree orthogonal groups ($\POmega_{2m+1}(q)=B_m(q)'$), and the exceptional groups. 
This leaves the even-degree orthogonal groups ($\POmega_{2m}^+(q)=D_m(q)'$ resp.\ $\POmega_{2m}^-(q)={}^2D_m(q^2)'$) of unbounded rank. 
\end{proof}
    
Singular mixed identities as in Theorem~\ref{thm:main_slfg}(iii) and Theorem~\ref{thm:main_fin}(ii) for nonsimply
laced root systems were constructed by Gordeev \cite{gordeev1997freedom}*{Theorem~3}. Theorem~\ref{thm:main_slfg} leads
us to Theorem~\ref{thm:strictsimpleintro}(i), which was already known for linear groups \cite{tomanov1985generalized}. For
nonlinear groups it follows from Theorem~\ref{thm:main_slfg}(i) as strict words are nonsingular.

By the above results, an infinite simple locally finite group that satisfies a nonsingular mixed identity must be as in Theorem~\ref{thm:main_slfg}(iii).

Note that the nonsingularity hypothesis in Conjecture~\ref{conj:main_lfsg}(ii) is necessary: e.g.\ both the finite alternating groups $\Alt_n$ and the groups $\PSL_n(q)$ (for any fixed prime power $q$) satisfy \emph{singular} mixed identities of bounded length (see \cite{bradfordschneiderthom2023length}*{Theorems 2 and 4}).

The following theorem partially establishes Conjecture~\ref{conj:main_lfsg}(ii):

\begin{theorem}\label{thm:part_conj_jones}
The finite alternating groups $\Alt_n$ and the groups from Theorem~\ref{thm:main_fin}(i), (iii), and (iv) together with the sporadic groups satisfy Conjecture~\ref{conj:main_lfsg}(ii). 
\end{theorem}

Indeed, it was shown in \cite{schneiderthom2022word} that the conjecture holds when we restrict to finite alternating groups. It is also valid for the nonclassical (twisted) groups of Lie type in Theorem~\ref{thm:main_fin}(i), (iii), and (iv), as they do not have short mixed identities at all. Concerning classical (twisted) groups of Lie type, we still need to prove that the length $l=\norm{w}$ of a shortest nonsingular mixed identity tends to infinity if the rank of the group tends to infinity. This was shown in \cite{bradfordschneiderthom2023non}*{Corollary~2} for groups of type $A_{n-1}(q)=\PSL_n(q)$. In Section~\ref{sec:classical_grps_of_lie_type} of this paper we prove the same but for the classical (twisted) groups of Lie type ${}^2A_{n-1}(q^2)=\PSU_n(q^2)$, $B_m(q)'=\POmega_{2m+1}(q)$, $C_m(q)=\Sp_{2m}(q)$, $D_m(q)'=\POmega_{2m}^+(q)$, ${}^2D_m(q^2)'=\POmega_{2m}^-(q)$ (see Theorem~\ref{thm:rank_bound}). Hence Conjecture~\ref{conj:main_lfsg}(ii) holds for such groups which have a simply laced root system. Also, clearly we can exclude the finitely many sporadic groups from our considerations.

\section{Classical groups of Lie type}\label{sec:classical_grps_of_lie_type}

In this section we prove lower bounds on a shortest nonsingular mixed identities for groups of type $\PSp_{2m}(q)$, $\POmega_{2m}^{\pm}(q)$, $\POmega_{2m+1}(q)$, and $\PSU_n(q^2)$. Choose the seminorm $\norm{\bullet}\colon \GL(V)\to\mathbb{N}$ to be the \emph{projective rank seminorm} $\norm{g}\coloneqq\min\set{\rk(g-\lambda 1_V)}[\lambda\in K^\times]$ (where $K=\mathbb{F}_q$ in the bilinear case, and $K=\mathbb{F}_{q^2}$ in the unitary case, and $V=K^n$). 

\subsection{The bilinear and unitary case}

We start with the bilinear case.

\begin{theorem}\label{thm:orth_symp}
Let $V\cong\mathbb{F}_q^n$ be a vector space of finite dimension $n\in\mathbb{N}$ over the finite field $\mathbb{F}_q$ with $q$ elements. Either let $f\colon V\times V\to\mathbb{F}_q$ be a nondegenerate alternating bilinear form or let $Q\colon V \to\mathbb{F}_q$ 
be a nonsingular quadratic form and let $f\colon V\times V\to\mathbb{F}_q$ be the associated 
symmetric bilinear form. Set $C=\GI(V,f)$ resp.\ $\GI(V,Q)$ to be the full isometry group of $f$ resp.\ $Q$. Moreover, let $d,l,r\in\mathbb{N}$ and assume the word 
$$
w=x_{i(1)}^{\varepsilon(1)}c_1\cdots c_{l-1}x_{i(l)}^{\varepsilon(l)}\in C\ast\mathbf{F}_r
$$ 
to be reduced of length $l\geq 2$, i.e.\ we restrict to $c_0=c_l=1_V$. Let $G =\Sp(V,f)$ resp.\ $G=\Omega(V,Q)$ denote the symplectic resp.\ quasisimple orthogonal group with respect to $f$ resp.\ $Q$. Let $\set{u_1,\ldots,u_d}$ and $\set{w_1,\ldots,w_d}$ be $d$-element linearly independent subsets of $V$ such that $u_i\perp u_j$, $w_i\perp w_j$ with respect to $f$, 
and $Q(u_i)=Q(w_i)=0$, for all $i,j\in\set{1,\ldots,d}$. If $i(1)=i(l)$ and $\varepsilon(1) = -\varepsilon(l)$, then assume in addition that the linear spans of $\{u_1,\ldots,u_d\}$ and $\{w_1,\ldots,w_d\}$ intersect trivially and are mutually perpendicular with respect to $f$. Then, if $d$ satisfies
\begin{equation}\label{eq:symp_orth_cs}
d\leq\frac{1}{l}\cdot(\norm{w}_{\rm crit}/3-1),
\end{equation}
there exist $g_1,\ldots,g_r\in G$ such that $u_i.w(g_1,\ldots,g_r)=w_i$ for $i\in\{1,\ldots,d\}$.
\end{theorem}

\begin{remark}
Note that the hypothesis ``\emph{$Q(u_i)=Q(w_i)=0$ for all $i$}'' 
is vacuous except for $p=2$ in the orthogonal case.
\end{remark} 

For the proof we write $\norm{w}_k$ for the \emph{$k$-norm} of $w$, i.e.\ the number of occurrences of $x_k$ and $x_k^{-1}$ in $w$ ($k\in\set{1,\ldots,r}$). 

\begin{remark}\label{rmrk:Witt}
We recall that, by Witt's Lemma, the full isometry group $\GI(V)$ of a formed space $V$ of Witt index $m$ acts transitively on the set of ordered bases for $m$-dimensional totally isotropic subspaces of $V$. For $U \leq V$ a totally isotropic subspace of  dimension $(m-1)$, let $U'$ be a totally isotropic subspace of  dimension $m$ containing $U$. Considering the possible determinants (and, in the orthogonal case, spinor norms) of elements of $\GI(V)$ preserving $U'$ and fixing an ordered basis for $U$, we see that the quasisimple subgroup $G$ of $\GI(V)$ (that is, $G = \Sp(V)$; $\Omega(V)$ or $\SU(V)$ in the cases for which the form on $V$ is nondegenerate alternating bilinear; nonsingular quadratic or nondegenerate sesquilinear) acts transitively on the set of ordered bases for $(m-1)$-dimensional totally isotropic subspaces of $V$.
\end{remark}

\begin{proof}[Proof of Theorem~\ref{thm:orth_symp}]
We have to find vectors $v_{i,j}^{\varepsilon}$ for $i\in\{1,\ldots,d\}$, $j\in\{1,\ldots,l\}$, $\varepsilon\in\{\pm1\}$ and elements $g_1,\ldots,g_r\in G$ as in the proof of \cite{bradfordschneiderthom2023non}*{Theorem~1}. This means that
\begin{gather*}
u_1=v_{1,1}^{\varepsilon(1)}\stackrel{g^{\varepsilon(1)}_{i(1)}}{\mapsto}
v_{1,1}^{-\varepsilon(1)}\stackrel{c_1}{\mapsto} v_{1,2}^{\varepsilon(2)}\stackrel{g^{\varepsilon(2)}_{i(2)}}{\mapsto}\cdots\stackrel{g^{\varepsilon(l-1)}_{i(l-1)}}{\mapsto}v_{1,l-1}^{-\varepsilon(l-1)}\stackrel{c_{l-1}}{\mapsto}v_{1,l}^{\varepsilon(l)}\stackrel{g^{\varepsilon(l)}_{i(l)}}{\mapsto}v_{1,l}^{-\varepsilon(l)}=w_1\\
\vdots\\
u_d=v_{d,1}^{\varepsilon(1)}\stackrel{g^{\varepsilon(1)}_{i(1)}}{\mapsto}v_{d,1}^{-\varepsilon(1)}\stackrel{c_1}{\mapsto} v_{d,2}^{\varepsilon(2)}\stackrel{g^{\varepsilon(2)}_{i(2)}}{\mapsto}\cdots\stackrel{g^{\varepsilon(l-1)}_{i(l-1)}}{\mapsto}v_{d,l-1}^{-\varepsilon(l-1)}\stackrel{c_{l-1}}{\mapsto}v_{d,l}^{\varepsilon(l)}\stackrel{g^{\varepsilon(l)}_{i(l)}}{\mapsto}v_{d,l}^{-\varepsilon(l)}=w_d.
\end{gather*}
Actually, we only have to inductively define the vectors $v_{i,j}^{-\varepsilon(j)}$ for $j\in\{1,\ldots,l-1\}$ which determine the rest: The vectors $v_{i,j}^{\varepsilon(j)}$ are determined by the rule $v_{i,j+1}^{\varepsilon(j+1)}=v_{i,j}^{-\varepsilon(j)}.c_j$. Moreover, it holds that $v_{i,1}^{\varepsilon(1)}=u_i$ and $v_{i,l}^{-\varepsilon(l)}=w_i$, so that these vectors are also already fixed. Assume we have already chosen $v_{1,1}^+,v_{1,1}^-,\ldots v_{i,j-1}^+,v_{i,j-1}^-,v_{i,j}^{\varepsilon(j)}$ and we want to choose $v_{i,j}^{-\varepsilon(j)}$, where it is understood that $v_{i,0}^\varepsilon=v_{i-1,l}^\varepsilon$ for $i\in\{2,\ldots,d\}$ and $\varepsilon\in\{\pm1\}$. Write $U\coloneqq\gensubsp{u_1,\ldots,u_d}$ and $W\coloneqq\gensubsp{w_1,\ldots,w_d}$. To carry out the induction process, we define subspaces for $i\in\set{1,\ldots,d}$, $j\in\set{1,\ldots,l}$, $k\in\set{1,\ldots,r}$, and $\varepsilon\in\set{\pm1}$:
\begin{align*}
V_{i,j,k}^\varepsilon &\coloneqq\gensubsp{v_{i',j'}^\varepsilon}[i'=i\text{ and }j'\leq j, \text{ or }i'<i,\text{ and }i(j')=k],
\end{align*}
Define
\begin{align*}
V_k^\varepsilon\coloneqq\begin{cases} 
U & \text{if }(k,\varepsilon)=(i(1),\varepsilon(1))\text{ and } (k,\varepsilon)\neq(i(l),-\varepsilon(l))\\
W & \text{if }(k,\varepsilon)\neq(i(1),\varepsilon(1))\text{ and } (k,\varepsilon)=(i(l),-\varepsilon(l))\\
U+W & \text{if }(k,\varepsilon)=(i(1),\varepsilon(1))=(i(l),-\varepsilon(l))\\
\mathbf{0} &\text{otherwise.}
\end{cases}
\end{align*}
Then this vector $v_{i,j}^{-\varepsilon(j)}$ must be isotropic (i.e.\ $f(v_{i,j}^{-\varepsilon(j)},v_{i,j}^{-\varepsilon(j)})=0$ and $Q(v_{i,j}^{-\varepsilon(j)})=0$ in the orthogonal case for $p=2$) and orthogonal to and not contained in the subspace $Z\coloneqq V_{i,j-1,i(j)}^{-\varepsilon(j)}+V_{i(j)}^{-\varepsilon(j)}$, which we abbreviate by $v_{i,j}^{-\varepsilon(j)}\perpnotin Z$. Here again it is understood that for $i\in\{2,\ldots,d\}$, $k\in\{1,\ldots,r\}$, and $\varepsilon\in\{\pm1\}$, one has $V_{i,0,k}^\varepsilon=V_{i-1,l,k}^\varepsilon$. Moreover, $V_{1,0,k}^\varepsilon=\mathbf{0}$. We have that $\dim(Z)\leq d\norm{w}_{i(j)}-1\leq dl-1$. Thus we have this many linear equations (linear forms) which need to hold for $v_{i,j}^{-\varepsilon(j)}$.
Next we have to ensure also that
$$
v_{i,j}^{-\varepsilon(j)}.c_j=v_{i,j+1}^{\varepsilon(j+1)}\perpnotin V_{i,j,i(j+1)}^{\varepsilon(j+1)}+V_{i(j+1)}^{\varepsilon(j+1)}\eqqcolon Z'.
$$
\vspace{3mm}
		
\noindent\emph{Case~1: Either $j\in J_0(w)$ or $j\in J_+(w)$.} For $j\in J_0(w)$, $i(j)\neq i(j+1)$, so $V_{i,j,i(j+1)}^{\varepsilon(j+1)}=V_{i,j-1,i(j+1)}^{\varepsilon(j+1)}$, which is already defined. Hence $\dim(Z')\leq d\norm{w}_{i(j+1)}-1\leq dl-1$.
		
For $j\in J_+(w)$, we have $i(j)=i(j+1)$ and $\varepsilon(j)=\varepsilon(j+1)$, so $V_{i,j,i(j+1)}^{\varepsilon(j+1)}=V_{i,j,i(j)}^{\varepsilon(j)}$ which is also already defined. Then again $\dim(Z')\leq d\norm{w}_{i(j)}-1\leq dl-1$.
Altogether, we must have
$$
v_{i,j}^{-\varepsilon(j)}\perpnotin Z\cup Z'.c_j^{-1}\quad\text{and}\quad v_{i,j}^{-\varepsilon(j)}\perp v_{i,j}^{-\varepsilon(j)} 
\quad\text{and}\quad Q(v_{i,j}^{-\varepsilon(j)}) = 0
$$
(the third condition being required only in the symmetric case and for $p=2$, 
in which case the second condition is vacuous).

By Warning's second theorem (see \cite{warning} or \cite{asgarli2018new}), 
there are at least 
$$
q^{n-2(dl-1)-2}-2q^{dl-1}=q^{n-2dl}-2q^{dl-1}\geq q^{n-2dl}-q^{dl}
$$
possible choices for $v_{i,j}^{-\varepsilon(j)}$. This is positive when $n-2dl>dl$, i.e.\ $n>3dl$ which holds by the Assumption~\eqref{eq:symp_orth_cs} in the theorem as $3dl<3(dl+1)\leq\norm{w}_{\rm crit}\leq n$.
		
\vspace{3mm}
		
\noindent\emph{Case~2: We have $j\in J_-(w)$.} Then
\begin{align*}
Z'=V_{i,j,i(j+1)}^{\varepsilon(j+1)}+V_{i(j+1)}^{\varepsilon(j+1)}&=V_{i,j,i(j)}^{-\varepsilon(j)}+V_{i(j)}^{-\varepsilon(j)}\\
&=V_{i,j-1,i(j)}^{-\varepsilon(j)}+\langle v_{i,j}^{-\varepsilon(j)}\rangle+V_{i(j)}^{-\varepsilon(j)}.
\end{align*}
Recall that $j\in\{1,\ldots,l-1\}$ and that $Z=V_{i,j-1,i(j)}^{-\varepsilon(j)}+V_{i(j)}^{-\varepsilon(j)}$. Then still $\dim(Z)\leq d\norm{w}_{i(j)}-1\leq dl-1$.
In total, we must ensure that
$$
v_{i,j}^{-\varepsilon(j)}\perpnotin Z\cup (Z+\langle v_{i,j}^{-\varepsilon(j)}\rangle).c_j^{-1}\quad\text{and}\quad v_{i,j}^{-\varepsilon(j)}\perp v_{i,j}^{-\varepsilon(j)} \quad\text{and}\quad Q(v_{i,j}^{-\varepsilon(j)}) = 0
$$ 
(with the third condition, once again, only being a requirement 
for the symmetric case in characteristic $2$). 

Hence there are at least
$$
q^{n-\dim(Z)-\dim(Z)-2-2}-q^{\dim(Z)}-\card{\{v\in V \mid v.c_j\in Z+\langle v\rangle\}}
$$
choices for $v_{i,j}^{-\varepsilon(j)}$.
		
The equation $v.(c_j-\lambda)\in Z$ has $q^{\dim(Z)}\leq q^{dl-1}$ solutions $v\in V$ for $\lambda$ not an eigenvalue of $c_j$; and it has $q^{\dim(Z)+d_\lambda}\leq q^{dl-1+d_{\lambda}}$ solutions $v\in V$ if $\lambda\in\mathbb{F}_q$ is an eigenvalue of $c_j$ with geometric multiplicity $d_\lambda$. So we have to exclude at most
$$
\card{\{v\in V \mid v.c_j\in Z+\langle v\rangle\}}\leq(q-m)q^{dl-1}+\sum_{\lambda}{q^{dl-1+d_\lambda}}
$$
many vectors, where $m$ is the number of eigenvalues and $\lambda$ ranges over the eigenvalues of $c_j$. Let $d_{\max}\in\mathbb{N}$ be the maximal geometric multiplicity that occurs:
Then the last expression is bounded from above by 
$$
q^{dl-1}(q-m+mq^{d_{\max}}).
$$
So we need to make the following expression positive
\begin{align*}
q^{n-2(dl-1)-2\cdot 2}-q^{dl-1}(q-m+1+mq^{d_{\max}})\\
=q^{n-2dl-2}-q^{dl-1}(q-m+    1+mq^{d_{\max}})\\
=q^{dl-1}(q^{n-3dl-1}-(q-m+    1+mq^{d_{\max}}))\\
\geq q^{dl-1}(q^{n-3dl-1}-(0+    1+q\cdot q^{d_{\max}})), 
\end{align*}
which is positive if $n-3dl-1>d_{\max}+1$, i.e.\ $n-d_{\max}=\norm{c_j}\geq \norm{w}_{\rm crit}>3dl+2$, which holds by Assumption~\eqref{eq:symp_orth_cs}. 
By Witt's Lemma, and noting Remark \ref{rmrk:Witt}, 
we find elements $g_1,\ldots,g_r\in G$ such that 
$u_i.w(g_1,\ldots,g_r)=w_i$ for $i\in\set{1,\ldots,d}$.
\end{proof}
	
We proceed with the unitary case:
	
\begin{theorem}\label{thm:untry_cse}
Let $V\cong\mathbb{F}_{q^2}^n$ be a vector space of dimension $n\in\mathbb{N}$ over the finite field $\mathbb{F}_{q^2}$ with $q^2$ elements and $f\colon V\times V\to\mathbb{F}_{q^2}$ be a nondegenerate sesquilinear form. Set $C=\GU(V,f)$ to be the general unitary group of $f$. Moreover, let $d,l,r\in\mathbb{N}$ and assume the word 
$$
w=x_{i(1)}^{\varepsilon(1)}c_1\cdots c_{l-1}x_{i(l)}^{\varepsilon(l)}\in C\ast\mathbf{F}_r
$$ 
to be reduced of length $l\geq 2$. Let $G =\SU(V,f)$ denote the special unitary group with respect to $f$. Let $\{u_1,\ldots,u_d\}$ and $\{w_1,\ldots,w_d\}$ be $d$-element sets of linearly independent vectors such that $u_i\perp u_j$ and $w_i\perp w_j$ with respect to $f$ for all $i,j\in\{1,\ldots,d\}$. If $i(1)=i(l)$ and $\varepsilon(1) = -\varepsilon(l)$, then assume in addition that the linear spans of $\{u_1,\ldots,u_d\}$ and $\{w_1,\ldots,w_d\}$ intersect trivially and are perpendicular with respect to $f$. Then, if $d$ satisfies
\begin{equation}\label{eq:untry_cs}
d\leq\frac{\norm{w}_{\rm crit}-2q-1}{3l}  
\end{equation}
there exist $g_1,\ldots,g_r\in G$ such that $u_i.w(g_1,\ldots,g_r)=w_i$ for $i\in\{1,\ldots,d\}$.
\end{theorem}
	
\begin{proof}
We retain the notations from the proof of Theorem~\ref{thm:orth_symp}, 
and seek once again to construct the vectors $v_{i,j}^{-\varepsilon(j)}$. 
Having done so, the existence of the $g_i$ will follow from Witt's Lemma 
and Remark \ref{rmrk:Witt}, as before.

\emph{Case~1: Either $j\in J_0(w)$ or $j\in J_+(w)$.} Define $Z$ and $Z'$ as in the proof of Theorem~\ref{thm:orth_symp}. Then we must have
$$
v_{i,j}^{-\varepsilon(j)}\perpnotin Z\cup Z'.c_j^{-1}\quad\text{and}\quad v_{i,j}^{-\varepsilon(j)}\perp            v_{i,j}^{-\varepsilon(j)}.
$$
It also holds that $\dim(Z),\dim(Z')\leq dl-1$. Hence if the following expression is positive, there is a legal choice for $v_{i,j}^{-\varepsilon(j)}$:
$$
q^{2(n-2(dl-1)-(q+1))}-2q^{2(dl-1)}>0
$$
we need $n-2(dl-1)-q-1>dl-1$, which means $n\geq 3dl+q-1$. This holds since $n\geq\norm{w}_{\rm crit}\geq 3dl+2q+1>3dl+q-1$ by Assumption~\eqref{eq:untry_cs}.
		
\vspace{3mm}
		
\noindent\emph{Case~2: $j\in J_-(w)$.} As above we must have that
$$
v_{i,j}^{-\varepsilon(j)}\perpnotin Z\cup (Z+\langle v_{i,j}^{-\varepsilon(j)}\rangle).c_j^{-1}\quad\text{and}\quad v_{i,j}^{-\varepsilon(j)}\perp v_{i,j}^{-\varepsilon(j)}.
$$ 
Hence we can estimate the number of possible choices for $v_{i,j}^{-\varepsilon(j)}$ as follows:
\begin{align*}
& q^{2(n-\dim(Z)-\dim(Z)-2(q+1))}-q^{2\dim(Z)}-\card{\{v\in V \mid v.c_j\in Z+\langle v\rangle\}}\\
& \geq q^{2(n-2(dl-1)-2(q+1))}-q^{2(dl-1)}-q^{2(dl-1)}(q^2-m+mq^{2d_{\max}})\\
&\geq q^{2(n-2dl-2q)}-q^{2(dl-1)}(1+q^2 q^{2d_{\max}}),
\end{align*}
so we need $n-2dl-2q>dl+d_{\max}$ which means $\norm{w}_{\rm crit}=n-d_{\max}\geq 3dl+2q+1$.
\end{proof}
	
\subsection{Diameter estimates and a proof of Theorem~\ref{thm:rank_bound}}

For $S\subseteq G$, $\diam(S)$ denotes the maximum distance between two elements of $S$ under the distance induced by our seminorm on $G$. Observe that \cite{schneiderthom2022word}*{Remark~1} (in the alternating case), \cite{bradfordschneiderthom2023non}*{Theorem~1} (in the special linear case) together with Theorems~\ref{thm:orth_symp} resp.~\ref{thm:untry_cse} (in the symplectic and orthogonal resp.\ the unitary case) from above imply that the quasisimple groups $G$ in these theorems satisfy
\begin{equation}\label{eq:ineq}
\norm{w}_{\rm crit}\leq\norm{w}(a\diam(w(G^r))+b)
\end{equation}
for every word $w$ of length $\norm{w}=l\geq 2$. 

\begin{remark} \label{rmk:crit_length}
In Inequality~\eqref{eq:ineq}, $a,b>0$ may be taken to be absolute constants in all cases except for that of $G$ a unitary group, in which case we may take $b=b(q)=2q+1$.
\end{remark}

Let $w'$ be an elementary reduction of $w$, where we delete a shortest critical constant $c_j$ and reduce the resulting word. We have that
$$
\norm{\diam(w(G^r))-\diam(w'(G^r))}\leq 2\norm{w}_{\rm crit}.
$$
so by \eqref{eq:ineq} we get
\begin{align*}
\diam(w'(G^r))&\leq\diam(w(G^r))+2\norm{w}_{\rm crit}\\
&\leq\diam(w(G^r))+2\norm{w}(a\diam(w(G^r))+b).
\end{align*}
Therefore, we obtain
$$
\diam(w'(G^r))+\frac{b}{a}\leq\left(\diam(w(G^r))+\frac{b}{a}\right)\left(1+2a\norm{w}\right).
$$
Hence, if we iterate these elementary reductions and $w=w_0,\ldots,w_m$ is a chain of words such that $w_j$ is a reduction of $w_{j-1}$ ($1\leq j\leq m$) and $w_m\neq1$ is \emph{strict}, then we get
$$
\diam(w_j(G^r))+\frac{b}{a}\leq\left(\diam(w_{j-1}(G^r))+\frac{b}{a}\right)\left(1+2a\norm{w_{j-1}}\right).
$$
for all $j=1,\ldots,m$, and by iterating we obtain
\begin{align*}
\diam(w_m(G^r))+\frac{b}{a} &\leq \left(\diam(w_0(G^r))+\frac{b}{a}\right)\prod_{j=0}^{m-1}(1+2a\norm{w_j})\\
&\leq\left(\diam(w_0(G^r))+\frac{b}{a}\right)(1+2\norm{w_0})^m\\
&\leq\left(\diam(w(G^r))+\frac{b}{a}\right)(1+2\norm{w})^{\lfloor\norm{w}/2\rfloor}.
\end{align*}
This is since $\norm{w_{j-1}}-2\geq\norm{w_j}$ so that $m\leq\lfloor\norm{w}/2\rfloor$. Hence by Inequality~\eqref{eq:ineq} and the last estimate we get
\begin{align*}
n=\norm{w_m}_{\rm crit}&\leq\norm{w_m}(a\diam(w_m(G^r))+b)\\
n&\leq\norm{w}(a\diam(w(G^r))+b)(1+2\norm{w})^{\lfloor\norm{w}/2\rfloor}.
\end{align*}
Thus we obtain
\begin{equation}\label{eq:final_est}
\frac{1}{(1+2\norm{w})^{\lfloor\norm{w}/2\rfloor}\norm{w}}\leq\frac{a\diam(w(G^r))+b}{n}.   
\end{equation}
In particular, this implies that when $\norm{w}$ is bounded and $n\to\infty$, then either 
$$
\diam(w(G^r))\to\infty
$$ 
or we are in the unitary case and $q \to\infty$ (since the latter is the only situation in which $a$ or $b$ can be unbounded, as noted above). 

\begin{proof}[Proof of Theorem~\ref{thm:rank_bound}]
Write $l=\norm{w}$. The alternating-group case follows from the
computation in
\cite{schneiderthom2022word}*{Theorem~1(i) and Remark~1}.
Indeed, for a nonsingular mixed identity their estimate specializes to
$
 \frac{1}{n}
 \geq
 \frac{1}{2}
 \exp\left(-\frac{l}{2}\log(5l)\right),
$
and hence
$
 n\leq 2(5l)^{l/2}.
$
The same computation already gives the asserted lower bound. The case
of groups of type $A_n$ was established in
\cite{bradfordschneiderthom2023non}*{Corollary~2}.

We consider the remaining classical groups. In the notation of this
section, Equation~\eqref{eq:final_est} gives
\[
 \frac{1}
 {(1+2l)^{\lfloor l/2\rfloor}l}
 \leq
 \frac{a\diam(w(G^r))+b}{n}.
\]
Since $w$ is a mixed identity, $\diam(w(G^r))=0$, and therefore
$n\leq b\,l(1+2l)^{\lfloor l/2\rfloor}.$
As per Remark \ref{rmk:crit_length}, for the symplectic and orthogonal groups, the constants $a$ and $b$ are
absolute. In the unitary case we may take $b=2q+1$. On the other hand,
\cite{bradfordschneiderthom2023length}*{Lemmas 8.2 and 8.5} give
$l=\Omega(q)$ for every mixed identity of $\PSU_n(q^2)$, even with
constants in the ambient projective general linear group. Consequently,
$b=O(l+1)$ also in the unitary case.

Thus, in every remaining case, there is an absolute constant $C>0$
such that
$
 n\leq C\,l(l+1)(1+2l)^{\lfloor l/2\rfloor}.
$
Taking logarithms, we obtain
\[
 \log n
 \leq
 \log C+\log l+\log(l+1)
 +\frac{l}{2}\log(1+2l)
 \leq C'\,l\log(l+1)
\]
for an absolute constant $C'>0$. If $l\geq\log n$, the desired
conclusion is immediate. Otherwise,
$
 \log(l+1)\leq\log(\log n+1)=O(\log\log n),
$
and hence
\[
 l=\Omega\left(\frac{\log n}{\log\log n}\right).
\]
Finally, the degree of the natural classical module is comparable with
the Lie rank, so replacing the former by the rank parameter $n$ only
changes the implicit constant. This proves the theorem.
\end{proof}

We may now also deduce Theorem~\ref{thm:strictsimpleintro}(ii) from the Introduction.

\begin{proof}[Proof of Theorem~\ref{thm:strictsimpleintro}(ii)]
    For a contradiction, suppose that $(G_i)_{i\in\mathbb{N}}$ is a sequence 
    of pairwise nonisomorphic nonabelian finite 
    simple groups satisfying strict mixed identities $w_i \in G_i \ast \gensubgrp{x}$ 
    of fixed length. 
    By Theorem~\ref{thm:rank_bound} we may assume that all $G_i$ 
    are groups of Lie type of bounded rank. 
    Passing to a subsequence, we may further suppose that there is 
    a Lie type $\mathbf{G}$; $a \in \set{1,2,3}$ and distinct prime powers 
    $q_i$ such that for all $i$, $G_i = {}^a \mathbf{G}(q_i)'$. 
    The required contradiction then follows from Theorem~\ref{thm:main_fin} 
    (where in Case~(ii) of Theorem~\ref{thm:main_fin} we use the fact that $w_i$ 
    has no critical constants). 
\end{proof}
    
\section{Algebraic groups}\label{sec:alg_grps}

In this section we establish Theorem~\ref{thm:main_slfg}(ii)--(iv), which will imply Theorem~\ref{thm:main_fin}(i)--(iv) by applying a Schwartz-Zippel bound. We use the language of linear algebraic groups. Throughout we assume that $K \leq E \leq F$ are field extensions, with $F$ algebraically closed, and that $G=\mathbf{G}(E)$ or ${}^a\mathbf{G}(E)$ is a (twisted) linear simple algebraic group which splits over $E$ (with a twist $\alpha\in\Aut(E)$ of order $a\in\set{2,3}$). Let $V=\mathbf{V}(E)$ be an absolutely irreducible module of $\mathbf{G}(E)$ resp.\ ${}^a\mathbf{G}(E)$. For a variety $V\subseteq\mathbf{V}(F)$ write $V(E)$ for its $E$-points. For polynomials $\mathcal{P}\subseteq F[X_1,\ldots,X_n]$ write $V_E(\mathcal{P})$ for the variety of $E$-points on which the polynomials in $\mathcal{P}$ vanish.

\begin{definition}
For $V$ as above, and $v\in V$, $l\in V^\ast$ such that $v.l=0$, 
we define the set of $(v,l)$-\emph{small} constants for $G$ with respect to the module $V$ by $\Lambda_{v,l}(G,V)\coloneqq\set{c\in\GL(V)}[(v.c^g).l=0 \text{ for all }g\in G]$.

When $e_1,\ldots,e_n\in V=\mathbf{V}(E)$ is an ordered weight basis for the algebraic group $G=\mathbf{G}(E)$ with dual basis $e_1^\ast,\ldots,e_n^\ast\in V^\ast$, we abbreviate $\Lambda_{e_1,e_n^\ast}(G,V)$ by $\Lambda(G,V)$. The elements of $\Lambda(G,V)$ are called \emph{T-small} with respect to $V$ (after Tomanov~\cite{tomanov1985generalized}*{p. 37}).
\end{definition}

\subsection{Infinite fields}

First, we consider the case where $E=F$ is algebraically closed. Indeed, it would be enough to require that $E$ is infinite, since then $\mathbf{G}(E)$ is Zariski dense in $\mathbf{G}(F)$ (by \cite{borel2012linear}*{Corollary~18.3}). In this situation, we have the following main result. When $w\in\GL(\mathbf{V}(K))\ast\gensubgrp{x}$ is a word with constants, we may assume that $K$ is finitely generated over its prime field.

\begin{lemma}[\cite{tits1972free}*{Lemma~4.1}]\label{lem:tits_lem}
Let $K$ be a finitely generated infinite field and let $t\in K^\times$ be an element of infinite order. Then there exists a local field $L$ equipped with an absolute value $\norm{\bullet}\colon L\to\mathbb{R}_{\geq 0}$ such that $K\leq L$ and $\norm{t}\neq 1$. 
In particular, $\norm{\bullet}\colon L\to\mathbb{R}_{\geq 0}$ 
is nontrivial on $K^\times$.
\end{lemma}

\begin{theorem}[Tomanov~\cite{tomanov1985generalized} extended]\label{thm:no_mxd_ids_qsmpl_alg_grps}
Suppose either (i) $K=\mathbb{F}_q$ for some prime power $q$ and $F\coloneqq\overline{\mathbb{F}}_q$, or (ii) $K$ is infinite finitely generated, 
$L$ and $\norm{\bullet}\colon L\to\mathbb{R}_{\geq 0}$ are as in the conclusion 
of Lemma~\ref{lem:tits_lem}, and 
$F\coloneqq\overline{L}$. Then $G=\mathbf{G}(F)$ admits no $\Lambda(\mathbf{G}(F),\mathbf{V}(F))$-noncentral mixed identity with constants in $\GL(\mathbf{V}(F))$.
\end{theorem}

The conclusion of Theorem~\ref{thm:no_mxd_ids_qsmpl_alg_grps} follows from the results of \cite{tomanov1985generalized} (see specifically the Proposition of \S 2 therein). 
That said we have not found the argument presented in \cite{tomanov1985generalized} 
to be entirely transparent, in the case for which $F$ is a locally finite field, 
since a key step is to construct a finitely generated Zariski dense 
subgroup of $G = \mathbf{G}(F)$ (see the first paragraph of the proof of Theorem~1 
in \cite{tomanov1985generalized}). 
As such, and owing to the centrality of this argument to our results, 
we provide here a clarifying argument.
    
\begin{proof}[Proof of Theorem~\ref{thm:no_mxd_ids_qsmpl_alg_grps}]
We start with Case~(i): $F=\overline{\mathbb{F}}_q$. Assume that $w(x)\in\GL_n(\mathbb{F}_q)\ast\gensubgrp{x}$ is such that $\im(w_G)\nsubseteq\mathbf{Z}(\GL(\mathbf{V}(F))$, where $w_G\colon G\to\GL(\mathbf{V}(F))$ is the word map on the group $G=\mathbf{G}(F)$. Then we prove that $w\colon\mathbf{G}(\mathbb{F}_q[X,X^{-1}])\to\GL(\mathbf{V}(\mathbb{F}_q[X,X^{-1}]))$ also hits a noncentral element: Write $w(x)=(p_{ij}(x))_{i,j=1}^n$ and define the polynomials $$q_{ij}\coloneqq p_{ii}-p_{jj}$$ for $i\neq j$. If $w$ was nontrivial on $\mathbf{G}(F[X,X^{-1}])$, there would be an element $g$ of this group, distinct indices $i,j\in\set{1,\ldots,n}$, and a Laurent polynomial 
$$p(g)\coloneqq p_{ij}(g)\neq 0 \quad\text{or}\quad p(g)\coloneqq q_{ij}(g)\neq 0.$$ 
But taking $d$ such that $\mathbb{F}_{q^d}\leq\overline{\mathbb{F}}_q$ is large enough, there is a quotient map $$\pi\colon\mathbb{F}_q[X,X^{-1}]\twoheadrightarrow\mathbb{F}_{q^d}$$ and $p(g).\pi\neq 0$.
Hence $w$ would not be trivial on $\mathbf{G}(\mathbb{F}_{q^d})\leq\mathbf{G}(\overline{\mathbb{F}}_q)$, contradicting the assumption.
Now $\mathbb{F}_q[X,X^{-1}]$ is dense in $L\coloneqq\mathbb{F}_q(\!(X)\!)$, which is a local field. Since $\mathbf{G}$ is generated by its root subgroups 
$\set{X_{\alpha}}[\alpha \in \Phi]$ 
(for $\Phi$ the root system associated to $\mathbf{G}$), 
we have that $\mathbf{G}(\mathbb{F}_q[X,X^{-1}])\subseteq\mathbf{G}(\mathbb{F}_q(\!(X)\!))$ is dense in the \emph{analytic} topology, i.e., the topology induced by the absolute value on $\mathbb{F}_q(\!(X)\!)$, since each $X_{\alpha} (\mathbb{F}_q[X,X^{-1}])$ 
is dense in $X_{\alpha}(\mathbb{F}_q(\!(X)\!))$ in the analytic topology. 
Since polynomials over $\mathbb{F}_q$ are continuous in the analytic topology,  $w\in\GL_n(\mathbb{F}_q)\ast\gensubgrp{x}$ is a mixed identity for $\mathbf{G}(L)$, 
and hence for $\mathbf{G}(\overline{L})$ 
(again by \cite{borel2012linear}*{Corollary~18.3}). 
We have therefore reduced Case~(i) to Case~(ii), 
with $K$ any infinite finitely generated subfield of $\mathbb{F}_q(\!(X)\!)$.

In Case~(ii), again $w\in\GL_n(\overline{L})\ast\gensubgrp{x}$ is a mixed identity for $\mathbf{G}(\overline{L})$. 
Let the action of $\mathbf{G}(\overline{L})$ on $\mathbf{V}(\overline{L})\cong \overline{L}^n$ admit an ordered weight basis $e_1,\ldots,e_n$ (where $e_1$ is the highest and $e_n$ the lowest weight vector) such that there is a semisimple element $g=\diag(\lambda_1,\ldots,\lambda_n)\in\mathbf{G}(\overline{L})$ with $\norm{\lambda_1}>\norm{\lambda_i}>\norm{\lambda_n}$ for $i\in\set{2,\ldots,n-1}$. By \cite{steinbergfaulknerwilson1967lectures}*{Theorem~39} there is such an element $g$.

Now in both cases, we can argue as in the proof of \cite{tomanov1985generalized}*{Lemma~1}: $w$ must have a critical constant from $\Lambda(\mathbf{G}(\overline{L}),\mathbf{V}(\overline{L}))\cap\GL(\mathbf{V}(F))=\Lambda(\mathbf{G}(F),\mathbf{V}(F))$. The latter equality holds, since $F$ is infinite, so that $\mathbf{G}(F)\subseteq\mathbf{G}(\overline{L})$ is Zariski dense.
\end{proof}

In the next subsection, we consider the two modules $V=\mathbf{V}(\gamma)$ for $\gamma\in\set{\alpha_0,\beta_0}$ a long or a short root of $G=\mathbf{G}(F)$. The results of Gordeev~\cite{gordeev1997freedom} help to give a reasonable description of the T-small constants in these cases.

\subsection{Computation of T-small constants}\label{subsec:cmp_sml_consts}

Recall that $F$ was an arbitrary algebraically closed overfield of $K$. Now we want to describe the set of intrinsic T-small elements $\Lambda(\mathbf{G}(F),\mathbf{V}(F))\cap\mathbf{G}(F)$ of the quasisimple group of Lie type $G=\mathbf{G}(F)$ with respect to certain irreducible modules $V=\mathbf{V}(F)$. Here we exploit the methods of Gordeev~\cite{gordeev1997freedom}.

\begin{definition}[\cite{gordeev1997freedom}*{Definition~1}]\label{def:gord_smll}		Fix a maximal torus $\mathbf{T}(F)\cong (F^\times)^d$, where $d=\rk(\mathbf{G})$ is the rank of $\mathbf{G}$. When $s\in G=\mathbf{G}(F)$ is semisimple, it can be conjugated into $\mathbf{T}(F)$, i.e.\ there is $g\in\mathbf{G}(F)$ such that $s^g\in\mathbf{T}(F)$. Call $s$ \emph{G-small}, if $s^g.\alpha=1$ for every long root $\alpha\colon\mathbf{T}(F)\to F^\times$. Collect these elements in the set $\Lambda_{\text{ss}}(\mathbf{G}(F))$.
\end{definition}

\begin{remark}
Clearly, every central element is G-small semisimple.
\end{remark}
	
\begin{definition}[\cite{gordeev1997freedom}*{Definition~2}]\label{def:sml_upts}
Let $\mathbf{G}$ have roots of two different length (i.e.\ it is not simply laced). Let $u\in G=\mathbf{G}(F)$ be a nontrivial unipotent element. 
\begin{enumerate}[(i)]
\item We say that $u$ is \emph{G-small} if $\overline{u^{\mathbf{G}(F)}}$ contains either no long or no short root elements. Collect these elements in the set $\Lambda_{\text{up}}(\mathbf{G}(F))$.
\item Such an element is defined to be of the \emph{first class} if $\overline{u^{\mathbf{G}(F)}}$ contains a long root element; and of the \emph{second class} if it contains a short root element. Collect these elements in the sets $\Lambda_{\text{up},1}(\mathbf{G}(F))$ and $\Lambda_{\text{up},2}(\mathbf{G}(F))$. 
\item An almost unipotent $c=zu$ (for $z\in\mathbf{Z}(\mathbf{G}(F))$ and $u\neq 1$ unipotent) is called \emph{G-small} of the \emph{first} or the \emph{second class} if and only if $u$ has this property.
\end{enumerate}
\end{definition}

\begin{remark}\label{rmk:conj_cls_rt_elts}
Note that, by the classification of unipotent conjugacy classes in simple algebraic groups (see \cite{spaltenstein2006classes}), if \eqref{eq:exc_cses} fails, then the Zariski closure of any unipotent conjugacy class in a simple algebraic group $G=\mathbf{G}(F)$ contains a long root element. Thus in this case, there are no G-small almost unipotents of the second class.
\end{remark}

\begin{remark}\label{rmk:lift_mxd_ids}
Also note that, when $\overline{G}$ is the simple quotient of $G=\mathbf{G}(F)$ (the adjoint group), then any mixed identity $$w=c_0 x_{i(1)}^{\varepsilon(1)}c_1\cdots c_{l-1} x_{i(l)}^{\varepsilon(l)}c_l\in\overline{G}\ast\mathbf{F}_r$$ for $\overline{G}$ lifts to a mixed identity $$\widetilde{w}=\widetilde{c}_0 x_{i(1)}^{\varepsilon(1)}\widetilde{c}_1\cdots \widetilde{c}_{l-1} x_{i(l)}^{\varepsilon(l)}\widetilde{c}_l\in G\ast\mathbf{F}_r$$ for $G$, where $\widetilde{c}_j\in G$ is a lift of $c_j\in\overline{G}$ ($j\in\set{0,\ldots,l}$), and we choose these lifts such that $\widetilde{w}(1_G)=\widetilde{c}_0\cdots\widetilde{c}_l=1_G$. Indeed, since the center $\mathbf{Z}(G)$ of $G$ is discrete and $G$ is connected, $\widetilde{w}(G^r)=\mathbf{1}$, as desired. So if we speak of a T-small (with respect to a module $V=\mathbf{V}(F)$) resp.\ G-small element $c\in\overline{G}$, we mean that any lift $\widetilde{c}\in G$ is T-small with respect to the chosen module $V=\mathbf{V}(F)$ resp. G-small. 
\end{remark}

The following lemma connects the notions of T- and G-smallness.
	
\begin{lemma}[\cite{gordeev1997freedom}*{Lemma~1}]\label{lem:gordeev_conn_sml_consts}
Let $\alpha_0$ be a long root and $\beta_0$ a short root of $\mathbf{G}$ and let $\mathbf{V}(\alpha_0)$ and $\mathbf{V}(\beta_0)$ denote the corresponding irreducible modules ($\mathbf{V}(\alpha_0)$ is the adjoint representation.) Write $c=su$ for the Jordan decomposition of the constant $c\in G=\mathbf{G}(F)$. We have the following implications:
\begin{enumerate}[(i)]
\item Whenever $c\in\Lambda(\mathbf{G}(F),\mathbf{V}(\alpha_0))\cap\mathbf{G}(F)$, it follows that $s\in\Lambda_{\text{ss}}(\mathbf{G}(F))$. Also then $u=1$ unless Condition~\eqref{eq:exc_cses} is fulfilled. 
\item If $c\in\Lambda(\mathbf{G}(F),\mathbf{V}(\beta_0))\cap\mathbf{G}(F)$, then $s\in\mathbf{Z}(\mathbf{G}(F))$. Moreover, $u\in\Lambda_{\text{up}}(\mathbf{G}(F))$ if $u\neq 1$, we assume that $\mathbf{G}$ is not simply laced, and Condition~\eqref{eq:exc_cses} is not met.
\end{enumerate}
\end{lemma}
	
\begin{remark}\label{rmk:smply_lced}
When $\mathbf{G}$ is simply laced, $\Lambda_{\rm ss}(\mathbf{G}(F))=\mathbf{Z}(\mathbf{G}(F))$, so that $G=\mathbf{G}(F)$ has no $\mathbf{Z}(\mathbf{G}(F))$-noncentral mixed identity. This means that $\overline{G}$ has no mixed identity at all (by Remark~\ref{rmk:lift_mxd_ids}).
\end{remark}
	
\begin{theorem}[\cite{gordeev1997freedom}*{Theorem~2}]\label{thm:gordeev_crit_consts}
Let $w\in G\ast\mathbf{F}_r=\mathbf{G}(F)\ast\mathbf{F}_r$ be a word with constants. Suppose that one of the following holds:
\begin{enumerate}[(i)]
\item $\mathbf{G}$ is simply laced and $w$ is $\mathbf{Z}(\mathbf{G}(F))$-noncentral;
\item $\mathbf{G}$ is not simply laced, the Condition~\eqref{eq:exc_cses} is not fulfilled, and 
$$
\Lambda_{\text{ss}}(\mathbf{G}(F))\cap I(w)=\varnothing\text{ or }\mathbf{Z}(\mathbf{G}(F))(\mathbf{1}\cup\Lambda_{\text{up}}(\mathbf{G}(F)))\cap I(w)=\varnothing;
$$
\item The Condition~\eqref{eq:exc_cses} is true and 
$$
\mathbf{Z}(\mathbf{G}(F))(\mathbf{1}\cup\Lambda_{\text{up},1}(\mathbf{G}(F)))\cap I(w)=\varnothing\text{ or }
$$ 
$$
\mathbf{Z}(\mathbf{G}(F))(\mathbf{1}\cup\Lambda_{\text{up},2}(\mathbf{G}(F)))\cap I(w)=\varnothing. 
$$
\end{enumerate}
Then $w$ is not a mixed identity for $G=\mathbf{G}(F)$.
\end{theorem}

\begin{remark}
Note that the statement of Theorem~2 in \cite{gordeev1997freedom} does not include 
the requirement that the G-small semisimple and unipotent constants appearing in $w$ must 
be \emph{critical} constants. Nevertheless, 
the proof of the result (see \S3 of \cite{gordeev1997freedom}, and especially item Equation~(3.1) therein) 
is based on the Tomanov's criterion (see Theorem~\ref{thm:no_mxd_ids_qsmpl_alg_grps} above). 
The conclusion of Theorem~\ref{thm:gordeev_crit_consts} 
is then immediate from Lemma~\ref{lem:gordeev_conn_sml_consts}.
\end{remark}

Note that Theorem~\ref{thm:gordeev_crit_consts}(i) and (ii) immediately imply the following positive answer to the main conjecture of Larsen--Shalev \cite{larsenshalev2022identities}.

\begin{corollary}\label{cor:comp_lie_grps} If $G$ is a simple compact real Lie group then every mixed identity 
$w\in G\ast\mathbf{F}_r$ for $G$ has a critical constant in $\mathbf{Z}(G)$. 
In particular, if $\mathbf{Z}(G)=\mathbf{1}$ then $G$ has no mixed identity.
\end{corollary}

\begin{proof}
We have that $G\leq\mathbf{G}(\mathbb{C})$ is the compact real form of $\mathbf{G}$ 
for a simple Lie type $\mathbf{G}$. By Zariski density \cite{borel2012linear}*{Corollary~18.3} every mixed identity $w\in G\ast\mathbf{F}_r$ for $G$ is also an identity for $\mathbf{G}(\mathbb{C})$, 
so Case~(i) or (ii) of Theorem~\ref{thm:gordeev_crit_consts} applies. In Case~(i) we are done. In Case~(ii) there must be a G-small almost unipotent element $c\in\mathbf{G}(\mathbb{C})$ which occurs as a critical constant in $w$. But then $c\notin G$, since $G$ is compact and hence contains no almost unipotents.
\end{proof}
	
In Table~\ref{tab:small-elements} below, we have computed the sets $\Lambda_{\text{ss}}(\mathbf{G}(F))$.
	
\begin{table}[htb]
\centering
\begin{tabular}{|c|c|c|c|}
\hline
$\mathbf{G}$ & long roots & $\Lambda_{\text{ss}}(\mathbf{G}(F))\cap\mathbf{T}(F)$ \\
\hline
$A_{n-1}=\SL_n$ & all & center\\
\hline
$B_m=\SO_{2m+1}$ & $\pm e_i \pm e_j$ ($D_m$) & $\{\diag(t,\ldots,t,1,t,\ldots,t)$\\ & ($i<j$) & $\mid t\in\set{\pm1}\}$\\
\hline
$C_m=\Sp_{2m}$ & $\pm 2e_i$ ($A_1^m$)& $\{\diag(t_1,\ldots,t_m,t_m,\ldots,t_1)$ \\ & ($i\in\set{1,\ldots,m}$) & $\mid t_i\in\set{\pm1}\}$\\
\hline
$D_m=\SO_{2m}$ & all & center\\
\hline
$G_2$ & 6 ($A_2$) & 
$\{(t,t^{-1}).h\mid t^3=1\}$,\\
& & where $h\colon (F^\times)^2\to (F^\times)^7$\\
\hline
$F_4$ & 24 ($D_4$) & $\set{(t,t,t,t).h}[t\in\set{\pm1}]$,\\
& & where $h\colon (F^\times)^4\to(F^\times)^{26}$\\
\hline
$E_6$ & all & center\\
\hline
$E_7$ & all & center\\
\hline
$E_8$ & all & center (trivial)\\
\hline
\end{tabular}
\caption{We present the sets $\Lambda_{\text{ss}}(\mathbf{G}(F))$ with respect to a lexicographically ordered weight basis of the natural module of the quasisimple group $\mathbf{G}(F)$; the map $h$ denotes a parametrization of the maximal torus.}
\label{tab:small-elements}
\end{table}

\subsection{Short mixed identities}\label{subsec:shrt_mxd_ids}

Subsequently, we will often use Table~\ref{tab:small-elements}. Fix the (twisted) linear quasisimple locally finite group $G={}^a\mathbf{G}(E)'$ for a locally finite field $E$.
In the next lemma, we determine which noncentral G-small semisimple elements lie in $G$ itself.

\begin{lemma}\label{lem:small_elts_in_fin_grps}
The following hold.
\begin{enumerate}[(i)]
\item If $\mathbf{G}\in\set{A_n,D_n,E_6,E_7,E_8}$, then $G\cap\Lambda_{\rm ss}(\mathbf{G}(\overline{\mathbb{F}}_q))=G\cap\mathbf{Z}(\mathbf{G}(\overline{\mathbb{F}}_q))\subseteq\mathbf{Z}(G)$.
\item Assume that $\mathbf{G}=C_n$ resp.\ $F_4$, so that ${}^a\mathbf{G}(E)$ is perfect. For $p\neq 2$ we must have $a=1$ in the case $\mathbf{G}=F_4$. Then it holds that $G\cap\Lambda_{\rm ss}(\mathbf{G}(\overline{\mathbb{F}}_q))\supsetneq\mathbf{Z}(G)$, since the corresponding elements in Table~\ref{tab:small-elements} all are fixed under the Frobenius. When $p=2$, \eqref{eq:exc_cses} holds and there is nothing to check. 
\item When $\mathbf{G}=B_m=\SO_{2m+1}$, we may assume that $p\neq 2$, since otherwise we are in Case~\eqref{eq:exc_cses} and there is nothing to do. We have that $G=B_m(E)'=\Omega_{2m+1}(E)$.  Then there is an element $c\in G$ which is conjugate in $\SO_{2m+1}(\overline{\mathbb{F}}_q)$ to the element 
$$
\diag(-1,\ldots,-1,1,-1,\ldots,-1)\in\SO_{2m+1}(\overline{\mathbb{F}}_q)
$$ 
from Table~\ref{tab:small-elements}, in the latter group.
\item If $\mathbf{G}=G_2$, then ${}^a\mathbf{G}(E)$ is perfect. If $p\neq 3$, then $a=1$, and we have $G\cap\Lambda_{\rm ss}(\mathbf{G}(\overline{\mathbb{F}}_q))\supsetneq\mathbf{Z}(G)=\mathbf{1}$ as above. Indeed, the element $(t,t^{-1}).h$ is conjugate to a noncentral element of order three of $G$ in $G_2(\overline{\mathbb{F}}_q)$. For $p=3$ we are in \eqref{eq:exc_cses} and there is nothing to do.
\end{enumerate}
\end{lemma}

\begin{proof}
Only (iii) and (iv) needs an explanation.

Regarding (iii), we may assume that $E=\mathbb{F}_q$ is itself finite by passing to a finite subfield. Choose a basis $e_1,\ldots,e_{2m+1}$ of $V=\mathbb{F}_q^{2m+1}$ such that $f(e_i,e_i)=1$ for all $i\in\set{1,\ldots,m,m+2,\ldots,2m+1}$ and $f(e_{m+1},e_{m+1})=\beta$, where $\beta$ is either $1$ or a fixed nonsquare $\alpha$. Then the element $c$ from above can be written as $r_{e_1}\cdots r_{e_m} r_{e_{m+2}}\cdots r_{e_{2m+1}}$, witnessing that it has determinant and spinor norm one. Here $r_v$ is the reflection with respect to the anisotropic vector $v\in V$.

For (iv) we recall that a semisimple element $s\in\mathbf{G}(\overline{\mathbb{F}}_q)$ is conjugate to an element in $\mathbf{G}(q)$ if and only if $s$ is conjugate to $s^\alpha$, where $\alpha\colon x\mapsto x^q$ is the standard Frobenius. But $s^\alpha=((t,t^{-1}).h)^\alpha=(t^q,t^{-q}).h$.
When $q\equiv 1$ modulo $3$, then the latter is equal to $(t,t^{-1}).h$ as desired.
If $q\equiv -1$ modulo $3$, then $(t^q,t^{-q}).h=(t^{-1},t).h$ which is conjugate to $(t,t^{-1}).h$ by an element of the Weyl group.
\end{proof}

In \cite{tomanov1985generalized} Tomanov gave examples of short mixed identities for the types $B_n$ and $C_n$ containing T-small critical constants with respect to the natural modules. It turned out that this construction can be generalized to the setting of linear simple algebraic groups. Indeed, Lemma~\ref{lem:small_elts_in_fin_grps} gives us enough G-small semisimple constants to construct Gordeev's mixed identities (since long and short root elements are always contained in $G$, which play the role of G-small unipotent elements of first and second class):
	
\begin{theorem}[Gordeev~\cite{gordeev1997freedom}*{Theorem~3}]\label{thm:gordeev_short_ids}
Assume that $\mathbf{G}$ (the adjoint type) is not simply laced, i.e.\ of type $B_n$, $C_n$, $F_4$ or $G_2$.
Then $G={}^a\mathbf{G}(E)'$ has a mixed identity $w$ with constants in $G$. The constants of $w$ may be taken to consist of \emph{any} G-small 
semisimple element of $G$; any long root element of $G$ 
and, in Case~\eqref{eq:exc_cses} only, any short root element of $G$.
\end{theorem}
	
From Theorem~\ref{thm:gordeev_short_ids} we deduce the first part of Theorem~\ref{thm:main_slfg}(iii).


\subsection{Finite groups of Lie type}

The statements of Theorem~\ref{thm:main_fin} can be deduced directly from the following Schwartz--Zippel estimate of Breuillard--Green--Guralnick--Tao \cite{breuillardgreenguralnicktao2015expansion}*{Proposition~5.4}. Recall that the (twisted) finite group ${}^a\mathbf{G}(q)$ (where $a\in\set{1,2,3}$) is defined to be the group of fixed points of a group automorphism $\delta\in\Aut(\mathbf{G}(\overline{\mathbb{F}}_q))$, so that $\delta^a$ is the $q$-Frobenius.

\begin{proposition} \label{BreGreGurTaoProp}
Let $\mathbf{G}(\overline{\mathbb{F}}_q)\leq\GL_n(\overline{\mathbb{F}}_q)$ be a connected simple linear algebraic group and set $G\coloneqq{}^a\mathbf{G}(q)$. Then let
$$G'=[G,G]\leq G={}^a\mathbf{G}(q)=\mathbf{G}(\overline{\mathbb{F}}_q)_\delta=\fix(\delta)\leq\mathbf{G}(\overline{\mathbb{F}}_q)$$ 
be a finite quasisimple subgroup. Let $p\colon\mathbf{M}_n(\overline{\mathbb{F}}_q)\rightarrow\overline{\mathbb{F}}_q$ be a polynomial which does not vanish identically on $\mathbf{G}(\overline{\mathbb{F}}_q)$. Set $V_{\mathbb{F}_q}(p)=\set{g \in \mathbf{M}_n(\overline{\mathbb{F}}_q)}[p(g)=0]$ 
and: 
$$
v(p,G')\coloneqq\frac{\card{V_{\mathbb{F}_q}(p)\cap G'}}{\card{G'}}.
$$
Then we have the Schwartz--Zippel bound $v(p,G')=O_{\mathbf{G}}(\deg(p)q^{-1/a})$.
\end{proposition}

\begin{proof}[Proof of Theorem~\ref{thm:main_fin}]
Let $\widetilde{G}\leq\mathbf{\widetilde{G}}(\overline{\mathbb{F}}_q)\leq\SL(\mathbf{V}(\overline{\mathbb{F}}_q))$ be linear quasisimple groups covering the simple groups $G\leq\mathbf{G}(\overline{\mathbb{F}}_q)$, with $\mathbf{V}$ the natural module for $\widetilde{G}$. 
Then the action of $\widetilde{G}$ on $\mathbf{V}(\overline{\mathbb{F}}_q)$ 
is irreducible, so that $\mathbf{Z}(\widetilde{G})$ acts by scalar matrices (by Schur's Lemma).
Choose $\widetilde{w}\in\mathbf{\widetilde{G}}(\overline{\mathbb{F}}_q)\ast\gensubgrp{x}$ of length $\norm{\widetilde{w}}=l=O_{\mathbf{G}}(q^{1/a})$ with a small enough implicit constant such that it descends to a word $w\in\mathbf{G}(\overline{\mathbb{F}}_q)\ast\gensubgrp{x}$ which maps $G$ to the identity, but is nontrivial on $\mathbf{G}(\overline{\mathbb{F}}_q)$. Let $\widetilde{w}'(x,y)\in\mathbf{\widetilde{G}}(\overline{\mathbb{F}}_q)\ast\gensubgrp{x,y}$ be the word obtained from $\widetilde{w}$ by substituting $y$ for every occurrence of $x^{-1}$. Consider the polynomial map $g\mapsto \widetilde{w}'(g,\adj(g)^\top)\eqqcolon(p_{ij}(g))_{i,j=1}^n$. Define $q_{ij}\coloneqq p_{ii}-p_{jj}$ for $i\neq j$. By assumption, there must be $i\neq j$ such that $p\coloneqq p_{ij}$ or $p\coloneqq q_{ij}$ is nonzero. 
Then by Proposition \ref{BreGreGurTaoProp}, applied to the polynomial $p\in\overline{\mathbb{F}}_q[X_{ij}]$, we have that $p$ does not vanish on the finite quasisimple group $\widetilde{G}$, when $\deg(p)=O_{\mathbf{G}}(q^{1/a})$, with sufficiently small implicit constant. As by definition $\deg(p)\leq (n-1)l$, this is fulfilled since by assumption 
$l=O_{\mathbf{G}}(q^{1/a})$ (again, with sufficiently small implicit constant), 
giving the desired contradiction.

We have just shown that, when $w\in\mathbf{G}(\overline{\mathbb{F}}_q)\ast\gensubgrp{x}$ is a mixed identity of length $O_{\mathbf{G}}(q^{1/a})$ for $G$, then it is also a mixed identity for $\mathbf{G}(\overline{\mathbb{F}}_q)$ and hence it must contain G-small critical constants as described in Theorem~\ref{thm:gordeev_crit_consts}.

In Cases~(i) and (iii) of Theorem~\ref{thm:main_fin}, $\mathbf{G}$ is simply laced and hence all G-small semisimple elements are central (cf.\ Table~\ref{tab:small-elements}) and hence trivial as $\mathbf{G}(\overline{\mathbb{F}}_q)$ is simple. Hence by Theorem~\ref{thm:gordeev_crit_consts}(i) $G$ has no mixed identities with constants in $\mathbf{G}(\overline{\mathbb{F}}_q)$ of length $O_{\mathbf{G}}(q)$ (provided the implicit constant is sufficiently small).

In Case~(ii) and (iv) of Theorem~\ref{thm:main_fin} there are mixed identities $w\in\mathbf{G}(\overline{\mathbb{F}}_q)\ast\gensubgrp{x}$ for $G$ of bounded length by Theorem~\ref{thm:gordeev_short_ids}.

In Case~(ii) the constants of $w$ can be taken from $G$ itself; see Theorem~\ref{thm:gordeev_short_ids}.
The stipulations on the constants in the cases for which \eqref{eq:exc_cses} holds or does not hold follow from Theorem \ref{thm:gordeev_crit_consts}. 

In Case~(iv), Condition \eqref{eq:exc_cses} applies. By Case~(ii), $w$ has G-small almost unipotent critical constants of both classes. 
It therefore suffices to show that $G$ itself contains 
no G-small unipotent elements of $\mathbf{G}(\overline{\mathbb{F}}_q)$. 
If $C$ is a conjugacy class in $\mathbf{G}(\overline{\mathbb{F}}_q)$, 
then the Zariski closure $\overline{C}$ of $C\subseteq\mathbf{G}(\overline{\mathbb{F}}_q)$ 
is invariant under $\mathbf{G}(\overline{\mathbb{F}}_q)$-conjugacy, so is a union 
of $\mathbf{G}(\overline{\mathbb{F}}_q)$-conjugacy classes. 
We may therefore define a partial order $\preceq$ 
on the set of $\mathbf{G}(\overline{\mathbb{F}}_q)$-conjugacy classes, 
where $C_1\preceq C_2$ 
if $C_1\subseteq\overline{C}_2$. 
We refer to the tables in \cite{spaltenstein2006classes}, 
where computations are made of the structure 
of this poset, when restricted to unipotent classes. 
Specifically, we are interested in Item ``$(b_6)$'' 
on p.\ 148 (which corresponds to $\mathbf{G}=G_2$); 
``$\Sp_4,p=2$'' on p.\ 233 
(corresponding to $\mathbf{G}=B_2$); 
and ``$X^G \; (p=2)$'' on p.\ 250 
(corresponding to $\mathbf{G}=F_4$). 
In all three cases, there is exactly 
one class each of long root elements $C_{\rm l}$ 
and short root elements $C_{\rm s}$ 
(and these classes are distinct). 
Both appear immediately above the unique 
$\preceq$-minimal class (which is the trivial class $\mathbf{1}$), 
and every other class $C\neq\mathbf{1}$ satisfies $C_{\rm l}\preceq C$ or $C_{\rm s}\preceq C$. 
Assume now that $u\in G$ is G-small unipotent of the first resp.\ second class. Then $C_{\rm l}\preceq u^{\mathbf{G}(\overline{\mathbb{F}}_q)}$ resp.\  $C_{\rm s}\preceq u^{\mathbf{G}(\overline{\mathbb{F}}_q)}$.
But $u\in\fix(\delta)$ and $\preceq$ is $\delta$-invariant, so that then $C_{\rm l}^\delta=C_{\rm s}\preceq(u^{\mathbf{G}(\overline{\mathbb{F}}_q)})^\delta=u^{\mathbf{G}(\overline{\mathbb{F}}_q)}$ resp.\ $C_{\rm s}^\delta=C_{\rm l}\preceq(u^{\mathbf{G}(\overline{\mathbb{F}}_q)})^\delta=u^{\mathbf{G}(\overline{\mathbb{F}}_q)}$. Hence, in both cases, we have $C_{\rm l},C_{\rm s}\preceq u^{\mathbf{G}(\overline{\mathbb{F}}_q)}$, so $u$ cannot be G-small of either class, as desired.
\end{proof}

\begin{remark}
Note that the bounds in Theorem~\ref{thm:main_fin} still depend on the dimension $n$ of the chosen module, which depends on the simple algebraic group $\mathbf{G}$. However, when $\mathbf{G}$ is exceptional, then $n$ is bounded and we obtain a lower bound $l=\Omega(q^{1/a})$ involving only an absolute constant. For the classical finite simple groups of Lie type, we presented some rank independent bounds in \cite{bradfordschneiderthom2023length}.
\end{remark}

\subsection{Infinite groups of Lie type}

Next, we deduce Theorem~\ref{thm:main_slfg}(ii)--(v).
Let $G={}^a\mathbf{G}(E)'$ be a (twisted) simple locally finite group over an infinite field $E\leq\overline{\mathbb{F}}_q$. If $w\in\mathbf{G}(\overline{\mathbb{F}}_q)\ast\gensubgrp{x}$ is a mixed identity for $G$, by Zariski density, it is also an identity for $\mathbf{G}(\overline{\mathbb{F}}_q)$. Hence Table~\ref{tab:small-elements} implies that Theorem~\ref{thm:main_slfg}(ii) and (iv) hold as there are no non-central G-small semisimple elements in this case. In Case~(iii) and (v) there are mixed identities for $G$ with G-small constants in $G$ itself by Theorem~\ref{thm:gordeev_short_ids}; as described in Theorem~\ref{thm:gordeev_crit_consts}.
In Case~(v) all the G-small almost unipotent elements lie outside the finite subgroups ${}^a\mathbf{G}(q)'$, which exhaust $G$, so that none of these constants can be contained in $G$.

\section{Nonlinear simple locally finite groups}\label{sec:nonlin_simple_loc_fin_grps}

It remains only to prove Theorem~\ref{thm:main_slfg}(i).

\subsection{Absence of nonsingular mixed identities}
	
In this subsection we establish that, in the \emph{nonlinear case}, the group $G$ has no nonsingular mixed identities (which is Theorem~\ref{thm:main_slfg}(i)). We start with some basic definitions and facts.
	
\begin{definition}\label{def:sec_cov}
Let $G$ be a locally finite group. A \emph{sectional cover} of $G$ is a family of pairs $(G_i,N_i)_{i\in I}$ indexed by the set $I$ such that
\begin{enumerate}[(i)]
\item $G_i$ is a finite subgroup of $G$ for all $i\in I$;
\item $N_i$ is a normal subgroup of $G_i$ for all $i\in I$, such that
\item for every finite subgroup $F\subseteq G$, there exists an $i\in I$, such that $F\subseteq G_i$ and $F\cap N_i=\mathbf{1}$.
\end{enumerate}
For $i\in I$, we call $G_i/N_i$ the \emph{factors} of the sectional cover. The sectional cover is called \emph{Kegel cover} if all of its factors are simple.
\end{definition}
	
The following result is classical and due to Kegel. For details, see \cite{hartley1995simple}*{Corollary 2.5}:
	
\begin{theorem}[Kegel, \cite{kegel}]\label{thm:kegel}
Every simple locally finite group admits a Kegel cover. 
\end{theorem}
	
There is a partial converse of Kegel's theorem which will be useful to verify whether a locally finite group is simple.
	
\begin{lemma}\label{lem:kegel_convrse}
If the locally finite group $G$ has a Kegel cover $(G_i,N_i)_{i\in I}$ such that $N_i$ is the unique maximal normal subgroup of $G_i$ (for all $i\in I$), then $G$ is simple.
\end{lemma}

For the proof of the previous lemma and further results of this paper, we define the structure of a directed set on $I$ by: $i<j$ if and only if $(G_i,N_i)\neq(G_j,N_j)$, $G_i\subseteq G_j$, and $N_j\cap G_i=\mathbf{1}$. Clearly, the corresponding relation $\leq$ is reflexive, antisymmetric, and transitive. Also, with this definition, $I$ is directed, as for $i,j\in I$ arbitrary, there is $k\in I$ such that $G_i,G_j\subseteq\gensubgrp{G_i,G_j}\subseteq G_k$ and $N_k\cap\gensubgrp{G_i,G_j}=\mathbf{1}=N_k\cap G_i=N_k\cap G_j$, so $i,j\leq k$.

The following is evident.

\begin{remark}
If $(G_i,N_i)_{i\in I}$ is a Kegel cover for $G$ and $J \subseteq I$ 
    is a subset which is cofinal with respect to the relation $<$, 
    then $(G_i,N_i)_{j\in J}$ is also a Kegel cover for $G$
\end{remark}
 
\begin{proof}[Proof of Lemma~\ref{lem:kegel_convrse}]
We may assume that $G$ is infinite, otherwise $G$ is a finite simple group (apply Definition~\ref{def:sec_cov}(iii) to $F=G$). Let $1_G\neq n\in G$. Then there is $i\in I$ such that $n\in G_i$. Now by (iii) of Definition~\ref{def:sec_cov} applied to $F=G_i$, since $G$ is infinite, there is $j>i$. Then $n\in G_i\setminus\mathbf{1}\subseteq G_j\setminus N_j$ for all such $j$, so that $\gennorsubgrp{n}\geq\gensubgrp{n^{G_j}}=G_j$. Moreover, the set of all such $j$ is cofinal in $I$: For $k\in I$ arbitrary, there is $j>i,k$, since $I$ is directed and $G$ is infinite. Hence $\gennorsubgrp{n}\supseteq\bigcup_{j}G_j=G$, i.e.\ $G$ is simple.
\end{proof}
	
To simplify the exposition, we need the following lemma, which proves that a nonabelian simple locally finite group has a Kegel cover with all factors nonabelian simple groups. We state it here without proof.
	
\begin{lemma}[\cite{hartley1995simple}*{Corollary 2.5}] Let $G$ be any infinite simple locally finite group. Then \begin{enumerate}[(i)]
\item The set of factors of any Kegel cover forms a set of finite simple groups of unbounded orders.
\item $G$ has a Kegel cover consisting of perfect groups. In particular, all factors are nonabelian finite simple groups. 
\end{enumerate}
\end{lemma}

Hence, subsequently, we will just say \emph{Kegel cover} (of an infinite simple locally finite group) for a Kegel cover without cyclic groups of prime order as a factor. 

We have the following easy result, which simplifies the subsequent analysis. It depends on the classification of finite simple groups. Write $\mathbf{G}$ for a Lie type and ${}^a\mathbf{G}$ for a twisted one.

\begin{lemma}\label{lem:red_one_ltype}
Let $G$ be an infinite simple locally finite group with Kegel cover $(G_i,N_i)_{i\in I}$. Then there exists a cofinal subset $J\subseteq I$, such that for all $j\in J$ one of the following holds:
\begin{enumerate}[(i)]
\item $G_j/N_j\cong \Alt_{n_j}$;
\item There is a classical Lie type $X\in\set{\PSL,\PSp,\POmega^\bullet,\PSU}$ such that for all $j\in J$ we have $G_j/N_j\cong X_{n_j}(q_j)$;
\item There exists a (possibly twisted) Lie type 
$$
{}^a\mathbf{G}\in\set{E_6,E_7,E_8,F_4,G_2,{}^2E_6,{}^3D_4',{}^2B_2,{}^2F_4,{}^2G_2}
$$ 
(here $a\in\set{1,2,3}$, where $a=1$ refers to the untwisted case) such that $G_j/N_j\cong{}^a\mathbf{G}(q_j)$.
\end{enumerate}
Here $n_j\geq 2$ are integers and $q_j$ suitable proper prime powers ($j\in J$).
\end{lemma}

\begin{proof}
Assume the opposite holds. Define the subsets 
\begin{gather*}
J_{\rm alt}\coloneqq\set{i\in I}[G_i/N_i\cong \Alt_{n_i}\text{ for }n_i\geq 5];\\
J_X\coloneqq\set{i\in I}[G_i/N_i\cong X_{n_i}(q_i)\text{ for }n_i\geq2\text{ and suitable } q_i];\\
J_{{}^a\mathbf{G}}\coloneqq\set{i\in I}[G_i/N_i\cong{}^a\mathbf{G}(q_i)\text{ for suitable } q_i].
\end{gather*}
of $I$, where $X$ is a classical Lie type and ${}^a\mathbf{G}$ is as above. Assume that all the sets $J_{\rm alt},J_X,J_{{}^a\mathbf{G}}$ are not cofinal. Then, by directedness and since $G$ is infinite, there is an index $j\in I$ such that any of the previous subsets intersects $I^+\coloneqq\set{i\in I}[i\geq j]$ trivially. Then $G_i/N_i$ is a sporadic group for all $i\in I^+$. But these Kegel factors are of unbounded order, whereas there exist only finitely many sporadic groups. This contradiction proves the claim.
\end{proof}

\begin{definition}\label{def:type}
Let $G$ be an infinite simple locally finite group.

\begin{enumerate}[(i)]
\item We say that $G$ is of \emph{type $\Alt_\infty$} if it admits a Kegel cover $(G_i,N_i)_{i\in I}$ such that
$$
G_i/N_i \cong \Alt_{n_i}
$$
for every $i\in I$. Since $G$ is infinite, the sequence $(n_i)_{i\in I}$ is necessarily unbounded.
\item Let $X\in\set{\PSL,\PSp,\POmega^\bullet,\PSU}$ be a family of classical finite simple groups of Lie type. We say that $G$ is of \emph{type $X_{\overline{n}}(\overline{q})$} if it admits a Kegel cover $(G_i,N_i)_{i\in I}$ such that
$$
G_i/N_i \cong X_{n_i}(q_i)
$$
for every $i\in I$, where
$$
\overline{n}\coloneqq\limsup_{i\in I} n_i,
\qquad
\overline{q}\coloneqq\limsup_{i\in I} q_i.
$$
\item Finally, for a (twisted) nonclassical family ${}^a\mathbf{G}$, we say that $G$ is of \emph{type ${}^a\mathbf{G}(\infty)$} if it admits a Kegel cover $(G_i,N_i)_{i\in I}$ such that
$$
G_i/N_i \cong {}^a\mathbf{G}(q_i)
$$
for every $i\in I$, where $(q_i)_{i\in I}$ must be unbounded, as $G$ is infinite.
\end{enumerate}
\end{definition}

\begin{remark}
By Lemma~\ref{lem:red_one_ltype}, every infinite simple locally finite group is of at least one of the types given in Definition~\ref{def:type}. It can happen that an infinite simple locally finite group is of some classical type $X_\infty(\infty)$ and at the same time it has type $X_\infty(q)$ for a prime power $q$: Let $q_i$ be a power of $q$ ($i\in I=\mathbb{N}$) with $\lim_{i\to\infty}{q_i}=\infty$ and consider the sequence $G_0=\PSL_{n_0}(q_0)\hookrightarrow\Alt(G_0)\hookrightarrow H_0=\PSL_{\card{G_0}}(q)\hookrightarrow G_1=\PSL_{\card{G_0}}(q_1)\hookrightarrow\Alt(G_1)\hookrightarrow H_1=\PSL_{\card{G_1}}(q)\hookrightarrow\cdots$. Then $(G_i,\mathbf{1})_{i\in I}$ and $(H_i,\mathbf{1})_{i\in I}$ are both Kegel covers of the same simple locally finite group $G$, where the former witnesses that $G$ is of type $\PSL_\infty(\infty)$ whereas the second shows that $G$ is of type $\PSL_\infty(q)$. 
\end{remark}

Next we characterize when a simple locally finite group is \emph{linear}, i.e.\ it embeds into some group $\GL_n(K)$ for some $n\in\mathbb{N}$ and some field $K$. If there exists such an embedding, we call $G$ \emph{linear of degree $n$}. For our characterization we need a deep result of Mal'cev:
	
\begin{theorem}[Mal'cev's representation theorem \cite{hartley1995simple}*{Theorem~2.7}]\label{thm:malcev}
Let $G$ be any group and $n\in\mathbb{N}$. Suppose that each finitely generated subgroup of $G$ is linear of degree at most $n$. Then $G$ is linear of degree $n$.    
\end{theorem}

Here is the promised characterization, which distinguishes the types $X_\infty(\overline{q})$, $X_n(\infty)$ (for $X$ a classical Lie type), and ${}^a\mathbf{G}(\infty)$ for $n\in\mathbb{N}$ and $\overline{q}$ arbitrary.
 
\begin{lemma}\label{lem:lin_lfsg}
Let $G$ be an infinite simple locally finite group. Then $G$ is linear if and only if $G$ is of type $X_n(\infty)$ ($X$ classical) or ${}^a\mathbf{G}(\infty)$ (${}^a\mathbf{G}$ as in Lemma~\ref{lem:red_one_ltype}(iii)) for some $n\in\mathbb{N}_{\geq2}$.
\end{lemma}
	
\begin{proof}
First, assume that $G$ is either of type $\Alt_\infty$ or $X_\infty(\overline{q})$ ($X$ classical). Then, for some Kegel cover $(G_i,N_i)_{i\in I}$ of $G$ we have $G_i/N_i\cong \Alt_{n_i}$ resp.\ $G_i/N_i\cong X_{n_i}(q_i)$ (for all $i\in I$) for an unbounded tuple $(n_i)_{i\in I}$. Suppose that $G$ is linear of degree $N$. Then not all $G_i$ can be a subgroup of $G$. Indeed, the simple factors $G_i/N_i\cong \Alt_{n_i}$ or $X_{n_i}(q_i)$ ($i\in I$) contain subgroups $H\subseteq G_i/N_i$ of the form $\Alt_n$ for arbitrarily large $n$ as $n_i$ tends to infinity. But the covers $\widetilde{H}\subseteq G_i$ of $H$ do not embed into $G$ for $n$ large enough, as they are not linear of degree $\leq N$.

Conversely, when $G$ is of type $X_n(\overline{q})$ resp.\ ${}^a\mathbf{G}(\overline{q})$, note that we must have $\overline{q}=\infty$, since otherwise $G$ would be finite. In this case, the Kegel factors $G_i/N_i$ stay linear of bounded degree:  E.g.\ $\PSL_n(q)\leq\GL_{n^2}(q)$. However, if $i\in I$, then $G_i$ is also linear of degree $N$ since for any $j>i$ (which exists since $G$ is infinite), we have the embedding $G_i\hookrightarrow G_j/N_j\leq\GL_N(q)$. Now Mal'cev's representation theorem gives that $G$ is linear of degree $N$.
\end{proof}

\subsubsection{A proof that Conjecture~\ref{conj:main_lfsg}(ii) implies Conjecture~\ref{conj:main_lfsg}(i)}

\begin{lemma}\label{lem:inf_loc_smpl_grp_non_trv_id}
Let $G$ be an infinite locally finite group admitting a mixed identity $w \in G\ast\mathbf{F}_r$ and let $(G_i,N_i)_{i\in I}$ be a sectional cover of $G$. Then, for all $n\in\mathbb{N}$ there is $j\in I$ such that for all $i \geq j$, $G_i/N_i$ has cardinality at least $n$, $w\in G_i\ast\mathbf{F}_r$, and the image $w_i=\overline{w}\in (G_i/N_i)\ast\mathbf{F}_r$ is a mixed identity for $G_i/N_i$ of the same content and length as $w$.
\end{lemma}
	
\begin{proof} 
Let $F$ be a finite subgroup of cardinality at least $n$ (which exists since $G$ is infinite) and $C\subseteq G$ be the constants in $w$. Then, the group $\langle C,F\rangle$ is finite and therefore contained in some $G_j$ ($j\in I$) such that $\langle C,F\rangle\cap N_j=\mathbf{1}$. Therefore for all $i \geq j$ we also have $\langle C,F\rangle \leq G_i$ and $\langle C,F\rangle\cap N_i=\mathbf{1}$. Hence $\langle C,F\rangle$ is also isomorphic to a subgroup of $G_i/N_i$. In this situation the finite  group $G_i/N_i$ is of cardinality at least $n$. Moreover, $w \in G_i\ast\mathbf{F}_r$ and it is clearly a mixed identity for $G_i$ by restriction. Hence the image of $w$ in $(G_i/N_i)\ast\mathbf{F}_r$ is a mixed identity for $G_i/N_i.$ The content is unchanged by this procedure. Since no constant of $w$ becomes trivial in $G_i/N_i$, the length is also unchanged. This finishes the proof.
\end{proof}
	
\begin{corollary}\label{cor:loc_fin_smpl_grp_no_id}
Conjecture~\ref{conj:main_lfsg}(ii) implies Conjecture~\ref{conj:main_lfsg}(i).
\end{corollary}
	
\begin{proof}
This directly follows from Theorems~\ref{thm:kegel} and Lemma~\ref{lem:inf_loc_smpl_grp_non_trv_id}.
\end{proof}

\subsubsection{Proof of the main result on nonlinear simple locally finite groups}

We are now prepared to prove Theorem~\ref{thm:main_slfg}(i): Let $G$ be of type $\Alt_\infty$ or $X_\infty(\overline{q})$ ($X$ classical) and assume $w\in G\ast\mathbf{F}_r$ is a nonsingular mixed identity of length $\norm{w}=l$. Then $G_i/N_i\cong \Alt_{n_i}$ or $X_{n_i}(q_i)$ ($i\in I$) for $n_i\to\infty$ along $I$. Hence by Theorem~\ref{thm:rank_bound} and Lemma~\ref{lem:inf_loc_smpl_grp_non_trv_id} we get a contradiction, since the mixed identity $w_i=\overline{w}\in (G_i/N_i)\ast\mathbf{F}_r$ must have length $l=\Omega\left(\frac{\log(n_i)}{\log\log(n_i)}\right)\to\infty$ along $I$.

\begin{example}[Hall's universal locally finite group]
Recall that Hall's universal locally finite group is constructed in \cite{hall1959some} as the ascending union $\bigcup_{i\in\mathbb{N}}{G_i}$, where $G_0=\Sym_3$ and $G_{i+1}=\Sym(G_i)\supseteq G_i$ via the regular representation of $G_i$. Then $G$ is simple and of type $\Alt_\infty$, which we briefly prove here:

For $i\in\mathbb{N}$, we have $$\Sym_{n_{i+1}}\cong \Sym_{n_i!}\cong G_{i+1}=\Sym(G_i)\supseteq G_i\cong \Sym_{n_i}.$$ Any element of $G_i=\Sym_{n_i}$ has order dividing $l_i\coloneqq\lcm\set{1,\ldots,n_i}$ and hence has a multiple of $c_i\coloneqq n_i!/l_i$ many cycles of the same length in $G_{i+1}=\Sym_{n_i!}$. For $n_i\geq 4$, $c_i$ is even, so that $$G_i\subseteq\Alt(G_i)\eqqcolon G_i'\subset\Sym(G_i)=G_{i+1}.$$ Hence a Kegel cover of $G$ is given by $(G_i',N_i'\coloneqq\mathbf{1})_{i\in\mathbb{N}_+}$, so that $G$ is simple (by Lemma~\ref{lem:kegel_convrse}) and of type $\Alt_\infty$. Theorem~\ref{thm:main_slfg}(i) now implies that $G$ has no nonsingular mixed identity.

However, one can also easily prove that $G$ has no mixed identities at all: Take $w\in G\ast\mathbf{F}_r$ nontrivial. Then $w\in G_{i+1}\ast\mathbf{F}_r=\Sym(G_i)\ast\mathbf{F}_r$ for some $i\in\mathbb{N}$ such that all constants in $w$ are isotypic. In particular, all critical constants must have full support (of size $\card{G_i}$) as they are nontrivial. By Inequality~\eqref{eq:ineq} we must have $\norm{w}_{\rm crit}=\card{G_i}\leq\norm{w}(a\diam(w(G_i^r))+b)$. But $a,b,\norm{w}$ are constant, whereas $\card{G_i}\to\infty$ as $i\to\infty$. Thus $\diam(w(G_i^r))\to\infty$ ($i\to\infty$) as well. Here the absence of mixed identities is basically due to the fact that $G$ is not \emph{finitary linear}, see Definition~\ref{def:ftry_lin} below.

This example shows that a nonlinear simple locally finite group does  not need to have singular mixed identities at all (cf.\ Theorem~\ref{thm:main_slfg}(i)). However, the infinite alternating group $\Alt(\Omega)$ ($\Omega$ an infinite set) has the singular mixed identity $w(x)=[\sigma,x]^{30}$ for $\sigma$ a $3$-cycle.
\end{example}

\subsection{Singular mixed identities}

In this subsection we prove that every nonlinear simple locally finite group which has a (necessarily singular) mixed identity is \emph{finitary linear} (see the definition below). Using the work of J.I.\ Hall~\cite{hall2006periodic}, we classify all cases which can occur, establishing the second part of Theorem~\ref{thm:main_slfg}(i). Let us start with a definition.

\begin{definition}\label{def:ftry_lin}
The group $G$ is called \emph{finitary linear} if and only if it embeds into $\FGL(V)\coloneqq\set{g\in\GL(V)}[\rk(1_V-g)<\infty]$ for some vector space $V$.
\end{definition}

We recall the following classification of finitary linear simple locally finite groups.

\begin{theorem}[J.I. Hall \cite{hall2006periodic}]\label{thm:hall}
A simple locally finite group that has a faithful representation as a finitary linear group is isomorphic to one of:
\begin{enumerate}[(i)]
\item a linear group in finite dimension;
\item an alternating group $\Alt(\Omega)$ with $\Omega$ infinite;
\item a finitary symplectic group $\FSp(V,f)$ over $E$;
\item a finitary special unitary group $\FSU(V,f)$ over $E$;
\item a finitary orthogonal group $\FOmega(V,Q)$ over $E$;
\item a finitary special linear group $\SL(V,W,m)$ over $E$.
\end{enumerate}

Here $E$ is a (possibly finite) subfield of $\overline{\mathbb{F}}_q$. The forms $f$ and $Q$ are nondegenerate on the infinite dimensional $E$-space $V$; and $m$ is a nondegenerate pairing of the infinite dimensional $E$-spaces $V$ and $W$. Conversely, each group in (ii)--$(vi)$ is simple, locally finite, and finitary but not linear in finite dimension.
\end{theorem}

Here comes the promised classification of nonlinear simple locally finite groups satisfying a mixed identity.

\begin{theorem}\label{thm:class_sing_mxd_id}
Let $G$ be an infinite nonlinear simple locally finite group. Then $G$ admits a singular mixed identity if and only if it is finitary linear and belongs to the following list:
\begin{enumerate}[(i)]
\item $G\cong \Alt(\Omega)$, where $\Omega$ is infinite;
\item $G$ is a finitary linear simple classical group over a finite field 
(that is, $G$ satisfies one of the conclusions (iii)--(vi) of Theorem~\ref{thm:hall}, 
with $E$ a finite field);
\item $G$ is a finitary symplectic group over an infinite locally finite field.
\end{enumerate}
\end{theorem}

\subsubsection{Reduction to finitary linear groups}

First, we show that every group $G$ as in Theorem~\ref{thm:class_sing_mxd_id} is finitary linear. Fix a Kegel cover $(G_i,N_i)_{i\in I}$ of $G$ with simple factors $Q_i\coloneqq G_i/N_i$ and natural maps $\pi_i\colon G_i\twoheadrightarrow Q_i$ ($i\in I$). We write $\norm{\bullet}_i\colon Q_i\to\mathbb{N}$ for the (unnormalized) Hamming norm in the alternating case ($Q_i\cong \Alt_{n_i}$), and for the projective rank norm in the classical case ($Q_i\cong X_{n_i}(q_i)$) on the group $Q_i$ ($i\in I$). The following observation is crucial:

\begin{lemma}\label{lem:propagation}
Assume that there are $1_G\neq c\in G$ and $d<\infty$ such that
$c\in G_i$, $c.\pi_i\neq 1_{Q_i}$, and $\norm{c.\pi_i}_i\leq d$ for all
sufficiently large $i$. Then for every $g\in G$ there is
$d(g)<\infty$ such that $\norm{g.\pi_i}_i\leq d(g)$ for all sufficiently
large $i$.
\end{lemma}

\begin{proof}
Since $G$ is simple and $c\neq 1_G$, the normal closure of $c$ is all of
$G$. Thus $g=c^{\varepsilon_1h_1}\cdots c^{\varepsilon_m h_m}$ for some
$h_1,\ldots,h_m\in G$ and signs
$\varepsilon_j\in\set{\pm 1}$ ($j\in\set{1,\ldots,m}$). For all large $i$, all elements involved lie
in $G_i$, so $g.\pi_i=
(c.\pi_i)^{\varepsilon_1(h_1.\pi_i)}
\cdots
(c.\pi_i)^{\varepsilon_m(h_m.\pi_i)}$.
The norms are conjugacy invariant and subadditive, so
$\norm{g.\pi_i}_i\leq md\eqqcolon d(g)$.
\end{proof}


\begin{lemma}\label{lem:PFGL-FGL}
Let $V$ be an infinite-dimensional vector space over a field $K$. Then
the natural map
\[
  \FGL(V)=
  \{\,a\in \operatorname{GL}(V)\mid \rk(a-1_V)<\infty\,\}
  \longrightarrow \operatorname{PGL}(V)
\]
identifies $\FGL(V)$ with
\[
  \PFGL(V)\coloneqq
  \{\,\overline{a}\in \operatorname{PGL}(V)\mid\norm{\overline{a}}_{\rm prk}<\infty\,\}.
\]
\end{lemma}

\begin{proof}
The map is injective. Indeed, if a nontrivial scalar $\lambda 1_V$ lies
in $\FGL(V)$, then $\rk(\lambda 1_V-1_V)<\infty$. Since $V$ is
infinite-dimensional, this forces $\lambda=1$. The map is surjective onto $\PFGL(V)$: Let $\overline{a}\in \PFGL(V)$, and
choose a representative $a\in\GL(V)$. By definition, there
is a scalar $\lambda\in K^\times$ such that
$\rk(a-\lambda 1_V)<\infty$. Then $\lambda^{-1}a\in \FGL(V)$ and maps
to $\overline{a}$. The lift is unique by injectivity, and therefore the
identification is an isomorphism of groups.
\end{proof}

\begin{proposition}\label{prop:ultraproduct-finitary}
Suppose that $G$ is a nonlinear infinite simple locally finite group, such
that the conclusion of Lemma~\ref{lem:propagation} holds. Then
$G$ admits a faithful finitary linear representation over a field.
\end{proposition}

\begin{proof}
Let $\mathcal U$ be a cofinal  ultrafilter on $I$. First suppose that $Q_i\cong \Alt_{n_i}$, with $n_i\to\infty$. Let
$V_i=\mathbb F_2^{n_i}$, and represent $Q_i$ by permutation matrices on
$V_i$. Put $V=\prod_{\mathcal U}V_i$, a vector space over
$\mathbb F_2$. Since $n_i\to\infty$, the space $V$ is
infinite-dimensional. For $g\in G$, define $g.\rho_i$ to be $g.\pi_i$ for all sufficiently large $i$, and put $g.\rho_i=1$ at
the remaining indices. Since $G=\bigcup_{i\in I}G_i$ and $N_i\cap G_j=1$ for
all large $i$, the relation $(gh).\rho_i=(g.\rho_i)(h.\rho_i)$ holds for
$\mathcal U$-almost all $i$, for every fixed pair $g,h\in G$. Thus the
ultraproduct gives a representation $\rho\colon G\to\operatorname{GL}(V)$. It is
faithful, because a nontrivial element of $G_j$ has nontrivial image in
all $Q_i$ with $i$ sufficiently large. Then $\rk(g.\rho_i-1)\leq d(g)$ for
$\mathcal U$-almost all $i$. By  {\L}o{\'s}' theorem~\cite{los1955remarques},
$\rk(g.\rho-1)\leq d(g)$, so $G.\rho\leq \FGL(V)$.

Now suppose that $Q_i\leq\operatorname{PGL}(V_i)$ are classical groups in
their natural projective representations, with $\dim(V_i)\to\infty$, where
$V_i$ is over the finite field $\mathbb{F}_{q_i}$. Set
$K=\prod_{\mathcal U}\mathbb{F}_{q_i}$ and $V=\prod_{\mathcal U}V_i$. Then $K$ is a
field and $V$ is an infinite-dimensional vector space over $K$. The same
construction as above gives a faithful projective representation $\overline{\rho}:G\to\operatorname{PGL}(V)$: choosing representatives in
$\operatorname{GL}(V_i)$ is harmless, because changing representatives
changes the ultraproduct representative by a scalar in $K^\times$.  {\L}o{\'s}' theorem~\cite{los1955remarques}, gives a projective rank norm of at most $d(g)$ for $g.\overline{\rho}$. Hence
$G.\overline{\rho}\leq\PFGL(V)$. By Lemma~\ref{lem:PFGL-FGL}, this lifts
uniquely to a faithful representation $G\hookrightarrow\FGL(V)$.
\end{proof}

\begin{remark}
The ultraproduct field $K$ need not itself be locally finite. Hall's theorem is
formulated for simple locally finite groups admitting a faithful finitary
linear representation over an arbitrary field. Its conclusion identifies the
classical finitary linear alternatives over a field $E\leq \overline{\mathbb F}_q$,
hence over a locally finite field; in characteristic $0$, the only
finitary linear simple locally finite alternative is the alternating group
\cite{hall2006periodic}.
\end{remark}

\begin{corollary}\label{cor:unbounded-finitary}
If $G$ admits a singular mixed identity and is of type $\Alt_\infty$ or $X_\infty(\overline{q})$, then $G$ is
finitary linear.
\end{corollary}

\begin{proof}
Let $w\in G\ast\mathbf{F}_r$ be a singular mixed identity. By
Lemma~\ref{lem:inf_loc_smpl_grp_non_trv_id}, the image
$w_i=\overline{w}\in Q_i\ast\mathbf{F}_r$ of $w$ is a mixed identity of the same length and
content for all $i$ in a cofinal subset of $I$. Equip $Q_i$ with the length function $\norm{\bullet}_i\colon Q_i\to\mathbb{N}$ as defined preceding Lemma~\ref{lem:propagation}. Inequality~\eqref{eq:ineq} gives an upper bound, 
depending only on $w$, for the value $\norm{c_j}_i$ for one critical constant $c_j$ of $w_i$.
As $w$ has only finitely many constants,
after passing to a further cofinal subset there are a critical constant
$c\in G$ and $d<\infty$ such that $\norm{c.\pi_i}_i\leq d$. 

If $G$ is of type $\PSU_\infty(\infty)$, 
so that we may take $Q_i = \PSU_{n_i}(q_i)$ with $q_i\to\infty$, 
then $G$ has no mixed identity by Lemma~\ref{lem:inf_loc_smpl_grp_non_trv_id}, 
since by \cite{bradfordschneiderthom2023length}*{Lemmas~8.2 and 8.5} the length of a shortest mixed identity for $Q_i$ is $\Omega (q_i)$. 
In all other cases, by Remark \ref{rmk:crit_length} 
we may take $d$ to be an absolute constant. 
Lemma~\ref{lem:propagation} gives a bound $\norm{g.\pi_i}_i\leq d(g)$ 
for every element $g$ of $G$.
Proposition~\ref{prop:ultraproduct-finitary} gives the required finitary
linear representation.
\end{proof}

\subsubsection{The finite field case}\label{subsubsec:fin_field_case}

Next we construct mixed identities in the setting of Theorem~\ref{thm:class_sing_mxd_id}(ii): Assume $E=\mathbb{F}_q$ is finite. 
Let $c\in G\leq\FGL(V)$ be nontrivial of order $e$. By \cite{bradfordschneiderthom2023non}*{Lemma~1} there exists $W\leq\fix(c)=\ker(c-1_V)$ of finite codimension and $U\leq V$ a $c$-invariant complement, such that $c=\rest{c}_U\oplus 1_W\in\SL(U)\oplus 1_W$ is of rank length $\rk(c-1_V)\leq \dim(U)\eqqcolon d$. We can choose $v\in\gensubgrp{x_1,\ldots,x_r}=\mathbf{F}_r$ as a law for the finite group $H\coloneqq\SL_{2dr}(q)$. 
Note that $[c,x_1] , \ldots , [c,x_r] \in G\ast\mathbf{F}_r$ freely generate a 
free subgroup of $G\ast\mathbf{F}_r$. 
Thus $w\coloneqq v([c,x_1] , \ldots , [c,x_r])\in G\ast\mathbf{F}_r$ is a (nontrivial) singular mixed identity for $G$. 

Note that since the length of $v$, and hence of $w$, depends only on $q$ and the codimension of $W$ in $V$, the above argument also yields the following, which provides the final ingredient in the proof of Theorem~\ref{thm:OrthCaseImpliesConj} from Section \ref{sec:def}.

\begin{proposition} \label{BddFieldSizeProp}
Any family of finite simple groups of Lie type over fields of bounded size satisfies mixed identities of bounded length.
\end{proposition}

\begin{proof}
Since there are boundedly many exceptional finite simple groups over fields of bounded size, it suffices to consider classical groups. As noted above, it suffices to show that for each group $G = \SL_n(q)$, $\Sp_{2n}(q)$, $\SU_n(q^2)$ or $\Omega_n^\bullet(q)$ there is a noncentral element $c \in G$ whose $1$-eigenspace in the standard module for $V$ has codimension bounded independent of $n$. 
This is achieved by elementary transvections, symplectic and unitary transvections and Eichler transformations, respectively.
\end{proof}


\subsubsection{The infinite field case}

Next we establish Theorem~\ref{thm:class_sing_mxd_id}(iii). Let $E$ be infinite. Assume that $G$ is of classical type $Y$ over $E$. Then $G$ is a filtered union of the infinite finite-rank quasisimple groups $G_i\cong Y_{n_i}(E)$ ($i\in I$). 
Hence, when $G$ has a mixed identity $w\in G\ast\mathbf{F}_r$ and $i\in I$ is large enough such that all constants of $w$ are in $G_i\setminus\mathbf{Z}(G_i)$ (this holds for all $i$ sufficiently large as $G$ itself has trivial centre), $w_i=\overline{w}\in Q_i\ast\mathbf{F}_r$ is a mixed identity for $Q_i=G_i/\mathbf{Z}(G_i)$ of bounded length, by Lemma~\ref{lem:inf_loc_smpl_grp_non_trv_id}. By Theorem~\ref{thm:main_fin} these $Q_i$ only satisfy a mixed identity of bounded length if $Q_i$ is one of $\POmega_{2m_i+1}(E)=B_{n_i}(E)$ or $\Sp_{2m_i}(E)=C_{n_i}(E)$. Hence $G$ must be of type $Y\in\set{\FSp,\FOmega}$.

In the symplectic case, when $p\neq 2$, the G-small semisimple constants from Table~\ref{tab:small-elements} may be chosen with bounded rank length; for example, in a hyperbolic basis, one may change signs on one hyperbolic pair and act trivially on the orthogonal complement. Such an element remains G-small
under the standard embeddings $$\Sp_{2m}(E)=C_m(E)\hookrightarrow C_n=\Sp_{2n}(E).$$ The same applies to long and short root elements also for $p=2$. Gordeev's identities~\cite{gordeev1997freedom} (see Theorem~\ref{thm:gordeev_short_ids} above) are therefore compatible with the direct limit and provide a singular mixed identity on $\FSp(V,f)$.

In the remaining case, $G$ must be of type $\FOmega(V,Q)$ and it must be a direct union of the quasisimple groups $G_i=\Omega_{2m_i+1}(E)$ ($i\in I$). As $\Sp_{2m}(E)\cong\Omega_{2m+1}(E)$ for $E$ locally finite of characteristic $p=2$ (since $E$ is perfect), we may assume that $p\neq 2$ as we just dealt with the symplectic case. Assume that $w\in G\ast\mathbf{F}_r$ is a (singular) mixed identity. Then, every constant $c$ of $w$ can be written as $c=\rest{c}_U\oplus 1_{U^\perp}$ for a nonsingular finite-dimensional space $U\leq V$. Hence they are contained in the orthogonal group $G_i=\Omega(W)$ (for $W\leq V$ finite-dimensional nondegenerate and some large $i$). By Theorem~\ref{thm:main_slfg}(iii) and Table~\ref{tab:small-elements}, there must be a critical constant $c$ in $w$ which is of the form $c=-1_{2m_i}\oplus 1\in G_i\subseteq G$. But then $w$ has critical constants conjugate to $-1_{2m_i}\oplus 1\oplus 1_{W^\perp}$ (for all large $i\in I$), which are infinitely many, a contradiction.

\begin{remark}
The nonexistence of \emph{nonsingular} mixed identities for $G$ the finitary symplectic; 
finitary special unitary or finitary orthogonal group of an infinite-dimensional 
formed space $V$ could also be 
proved using the methods of Marimon and Pinsker \cite{marimonpinsker2026all}. 
Using Witt's Lemma, one may show that the \emph{algebraic closure} 
of a finite set $A \subseteq V$ under the action of $G$ (see \cite{marimonpinsker2026all} 
Definition 2.7) is precisely the linear span of $A$. 
From this one deduces that $V$ is a \emph{higher Neumann modular pre-geometry} 
in the sense of \cite{marimonpinsker2026all}, such that Theorem~3.8 of that 
paper is applicable.
\end{remark}

\appendix

\section{An alternative proof}

We present here the following weaker version of Theorem~\ref{thm:no_mxd_ids_qsmpl_alg_grps}.
It relates the work of Golubchik--Mikhalev~\cite{golubchikmikhalev1982generalized}*{Section~1.2} with the approach of Tomanov~\cite{tomanov1985generalized}*{\S 2}, which is why we state it here. The linear operator $h_0=h^{m-1}$ in the proof below plays the same role as the linear operator $y$ in the first reference. Indeed, this justifies that the notion of T-smallness is natural to consider.

\begin{theorem}\label{thm:main_wo_neq2}
Assume $V=\mathbf{V}(F)$ to be a representation of $G=\mathbf{G}(F)$ with a nilpotent element $h\in\mathfrak{g}(F)\leq\End(\mathbf{V}(F))$ admitting a unique largest Jordan block. Let $p$ be large enough depending on $h$. Then $\mathbf{G}(F)$ has no $\Lambda_{v,l}(\mathbf{G}(F),\mathbf{V}(F))$-noncentral mixed identity with constants in $\GL(\mathbf{V}(F))$, for some $v\in\mathbf{V}(F)\setminus\mathbf{0}$ and $l\in\mathbf{V}^\ast(F)\setminus\mathbf{0}$ only depending on $h$.
\end{theorem}

\begin{proof}[Proof of Theorem~\ref{thm:main_wo_neq2}]
Let $h\in\mathfrak{g}(F)$ be a nilpotent element with a single largest Jordan block $J_m(0)$. Assume that $m\in\set{2,\ldots,\min(n,p)}$ and set $h_0=h^{m-1}$. Then, $h_0$ is of Jordan type $J_2(0)\oplus J_1(0)^{\oplus(n-2)}$. Thus it is of shape $lv$ for some $v\in\mathbf{V}(F)$ and $l\in\mathbf{V}^\ast(F)$ with $v.l=0$. We may assume that 
$$
w(x)=c_0x_{i(1)}^{\varepsilon(1)}c_1\cdots c_{l-1}x^{\varepsilon(l)}c_l\in\GL(\mathbf{V}(F))\ast\gensubgrp{x},
$$
where $c_j\notin\Lambda\coloneqq\Lambda_{v,l}(\mathbf{G}(F),\mathbf{V}(F))$ for all $j\in\set{1,\ldots,l-1}$ by applying the same substitution trick as in \cite{bradfordschneiderthom2023length}*{Lemma~2.2}. Assume that $c\notin\Lambda$. Then there is $g\in\mathbf{G}(F)$ such that $(v.c^g).l\neq 0$ by definition, and the set of all such elements $g$ is Zariski open and nonempty in $\mathbf{G}(F)$ (since we can write $\adj(g)^\top$ for $g^{-1}$). Hence we find $g\in\mathbf{G}(F)$ such that $	(v.c_j^g).l\neq 0$ holds simultaneously for all $j\in\set{1,\ldots,l-1}$. But the $c_j'=c_j^g\in\GL(\mathbf{V}(F))$ are precisely the intermediate constants in 
$$
w'(x)=w(gxg^{-1})=gx^{\varepsilon(1)}c_1^g\cdots c_{l-1}^g x^{\varepsilon(l)} g^{-1}=gx^{\varepsilon(1)}c_1'\cdots c_{l-1}'x^{\varepsilon(l)}g^{-1}.
$$ 
Replace $w$ by $w'$ and remove $c_0'=g$ and $c_l'=g^{-1}$ from it (as we may). Then the operators $h_0 c_j h_0$ (for $j\in\set{1,\ldots,l-1}$) have the same image and kernel as $h_0$ (and hence are equal to $\alpha_j h_0$ for some $\alpha_j\in F^\times$). Hence, one obtains that 
$$
h_0c_1h_0\cdots h_0c_{l-1}h_0=\beta h_0\neq 0
$$
for some $\beta\in F^\times$. But this operator $h_0c_1h_0\cdots h_0c_{l-1}h_0\in\End(\mathbf{V}(F))$ is precisely plus or minus the leading coefficient (of $\lambda^{(m-1)l}$) in the polynomial 
$$
v(\lambda)\coloneqq w\left(1+\lambda h+\frac{1}{2}\lambda^2 h^2+\cdots+\frac{1}{(m-1)!}\lambda^{m-1}h^{m-1}\right)\in\End(\mathbf{V}(F))[\lambda]
$$ 
up to a factor $(1/(m-1)!)^l\neq 0$. So \cite{bradfordschneiderthom2023length}*{Lemma~4.4} applies.
\end{proof}

In Theorem~\ref{thm:main_wo_neq2}, if $p\neq 2$, one can e.g.\ take the adjoint module $\mathbf{V}(\alpha_0)$ and for $h$ the operator $\ad(X_{\alpha_0})$ for the long root $\alpha_0$. Then $m=3$, since $h$ has a unique largest Jordan block of size three.

\section{Remarks on Schwartz--Zippel bounds}

For the interested reader, we give here a self-contained argument establishing Proposition~\ref{BreGreGurTaoProp} in the untwisted and the Steinberg case, providing an alternative proof for \cite{breuillardgreenguralnicktao2015expansion}*{Proposition~5.4}.

\subsection{The untwisted case}

Here we have the inclusion $G'\leq G=\mathbf{G}(q)\leq\mathbf{G}(\overline{\mathbb{F}}_q)$ of linear groups. By the standard Schwartz--Zippel type bound 
$$
\card{V(\mathbb{F}_q)}\leq\deg(V)q^{\dim(V)}
$$ holds for any affine variety $V\subseteq\mathbf{V}(\overline{\mathbb{F}}_q)$.
For $V$ we take the variety $V_{\overline{\mathbb{F}}_q}(p)\cap\mathbf{G}(\overline{\mathbb{F}}_q)$, which yields using B\'ezout's theorem
$$
\card{V(\mathbb{F}_q)}=\card{V_{\mathbb{F}_q}(p)\cap\mathbf{G}(q)}\leq\deg(p)\deg(\mathbf{G})q^{\dim(\mathbf{G})-1},
$$
since $\mathbf{G}(\overline{\mathbb{F}}_q)$ is connected and $p$ does not vanish on it by assumption, so that the intersection $V_{\overline{\mathbb{F}}_q}(p)\cap\mathbf{G}(\overline{\mathbb{F}}_q)$ has one dimension less than $\mathbf{G}(\overline{\mathbb{F}}_q)$. Hence
$$
v(p,G)=\frac{\card{V_{\mathbb{F}_q}(p)\cap G}}{\card{G}}\leq\deg(p)\deg(\mathbf{G})\frac{q^{\dim(\mathbf{G})-1}}{\card{G}}=O_{\mathbf{G}}(\deg(p) q^{-1}).
$$
As $G'$ has bounded index in $G$, we obtain $v(p,G')\leq v(p,G)[G:G']=O_{\mathbf{G}}(\deg(p)q^{-1})$ as required in Proposition~\ref{BreGreGurTaoProp}.

\subsection{The Steinberg case}

We recall that the \emph{generalized Steinberg group} ${}^a\mathbf{G}(E)=\mathbf{G}(E)_\delta=\fix(\delta)$ is the set of fixed points of an endomorphism $\delta=\sigma\alpha\in\End(\mathbf{G}(E))$, where $\alpha\in\End(E)$ is a field endomorphism and $\sigma$ is a diagram automorphism of order $a\in\set{2,3}$, both of which are suppressed in the notation. In the case that $\alpha$ is the Frobenius $x\mapsto x^q$ on $E=\overline{\mathbb{F}}_q$, we obtain the ordinary Steinberg groups. By restricting the field $E$, we may assume that $\alpha^a=\id$. Indeed $\fix(\delta)=\fix(\sigma\alpha)\subseteq\fix(\delta^a)=\fix(\alpha^a)=E_{\alpha^a}$, since $\alpha$ and $\sigma$ commute. Hence ${}^a\mathbf{G}(E)={}^a\mathbf{G}(E_{\alpha^a})$.

\medskip

The aim of the following argument is to write the finite group ${}^a\mathbf{G}(q^a)$ as the $\mathbb{F}_q$-points of a variety $V\subseteq\mathbf{M}_n(\overline{\mathbb{F}}_q)$. This variety will be the set $V_{\overline{\mathbb{F}}_q}(\mathcal{Q}')$ below.
    
Next, we describe the locally finite fields $E$ which admit an automorphism $\alpha$ of order $a$. The proof of the subsequent lemma is left to the reader.
	
\begin{lemma}\label{lem:ext_fld_aut}
Let $\mathbb{F}_{q^a}$ be a finite field with an automorphism $\alpha$ of order $a$. Then $\alpha$ can be extended to a field automorphism of the field $\mathbb{F}_{q^{ab}}$ of order $a$ if and only if $\gcd(a,b)=1$. If this condition is fulfilled, the extension is unique.
\end{lemma}
	
If $\mathcal{B}\subseteq\set{b\in\mathbb{Z}_+}[\gcd(a,b)=1]$ is a directed subset of $(\mathbb{Z}_+,\mid\,)$, we can construct a field $E=E_{\mathcal{B}}$ as the union of all fields $\mathbb{F}_{q^{ab}}$ ($b\in\mathcal{B}$). It has exactly $\varphi(a)$  many automorphisms of order $a$ which are all powers of each other. Fix one of them and call it $\alpha$. Write $E_\alpha=\fix(\alpha)=\bigcup_{b\in\mathcal{B}}\mathbb{F}_{q^b}$ for its fixed field. Conversely, every locally finite field $E$ admitting an automorphism of order $a$, is of the above type $E_{\mathcal{B}}$ for suitable $\mathcal{B}$ and $q$. When we require $\mathcal{B}$ to be downclosed, it is a one-to-one correspondence.

The following lemma shows that every infinite generalized Steinberg group ${}^a\mathbf{G}(E)$ is Zariski dense in the ambient algebraic group $\mathbf{G}(\overline{\mathbb{F}}_q)$.
	
\begin{lemma}\label{lem:twst_grp_zar_dense}
Recall that $\mathbf{G}$ is assumed to be $E$-split. Let $a\in\set{2,3}$. Then ${}^a\mathbf{G}(E)$ is Zariski dense in $\mathbf{G}(\overline{\mathbb{F}}_q)$ when $E$ is infinite.
\end{lemma}
	
\begin{proof}
The Lie algebra of the group ${}^a\mathbf{G}(E)$ is the $E_\alpha$-Lie subalgebra ${}^a\mathfrak{g}(E)$ of $\mathfrak{g}(E)$ generated by the elements $X_{\beta,\lambda}\coloneqq\lambda^{\alpha^0} X_{\beta.\sigma^0}+\cdots+\lambda^{\alpha^{a-1}}X_{\beta.\sigma^{a-1}}$ for $\lambda\in E$ and $X_\beta$ a root vector of $\mathfrak{g}(E)$ such that $\beta.\sigma\neq\beta$, and elements $\lambda X_\beta$ for $\lambda\in E_\alpha$ and $\beta\in\fix(\sigma)$, where $\sigma$ is the corresponding diagram automorphism of order $a$. Since $\overline{{}^a\mathbf{G}(E)}$ are the $\overline{\mathbb{F}}_q$-points of an algebraic group, its Lie algebra is an $\overline{\mathbb{F}}_q$-vector space. Hence, it must contain all elements $\mu X_{\beta,\lambda}$ for $\mu\in\overline{\mathbb{F}}_q$, $\lambda\in E$ and $\beta\neq\beta.\sigma$, and the elements $\mu X_\beta$ for $\beta\in\fix(\sigma)$.
The $X_{\beta,\lambda_i}$ for $i\in\set{1,\ldots,a}$ are $E$-linearly independent for suitable elements $\lambda_i$. We have
$$
\det\begin{pmatrix}
\lambda_1 & \cdots &\lambda_a\\
\vdots & \ddots & \vdots\\
\lambda_1^{q^{a-1}} & \cdots & \lambda_a^{q^{a-1}}
\end{pmatrix}\neq 0
$$
for some $\lambda_1,\ldots,\lambda_a\in\mathbb{F}_{q^a}\subset E$, as it is a homogeneous polynomial of degree $\frac{q^a-1}{q-1}<q^a$.
Hence the root element $\mu X_\beta$ (for all $\mu\in\overline{\mathbb{F}}_q$ and roots $\beta$) is contained in the Lie algebra of $\overline{{}^a\mathbf{G}(E)}$, so the latter must be all of $\mathbf{G}(\overline{\mathbb{F}}_q)$ by connectedness.
\end{proof}

Let $\alpha\colon\mathbb{F}_{q^a}\to\mathbb{F}_{q^a}$; $x\mapsto x^q$ be the $q$-Frobenius automorphism (of order $a$). Write $n$ for the dimension of the chosen module. Fix a set of finitely many \emph{twisted} polynomials
$$
\mathcal{Q}\subseteq \mathbb{F}_{q^a}[X_{ij}^{\alpha^k}\mid i,j\in\set{1,\ldots,n}, k\in\set{0,\ldots,a-1}]
$$ 
such that ${}^a\mathbf{G}(q^{ab})=V_{\mathbb{F}_{q^{ab}}}(\mathcal{Q})$ for all $b\in\mathbb{Z}_+$ with $b\equiv 1$ modulo $a$. (Here each variable that occurs, gets a formal exponent from $\gensubgrp{\alpha}$.) We extend the operator $V$ to twisted polynomials in the natural way.
	
Let us briefly explain how we get the finite set $\mathcal{Q}$: By definition, ${}^a\mathbf{G}(q^{ab})=\set{g\in\mathbf{G}(\overline{\mathbb{F}}_q)}[g.\delta_b=g]\subseteq\mathbf{G}(q^{ab})$, where $\delta_b=\alpha_b\sigma$ is the Steinberg endomorphism of $\mathbf{G}(\overline{\mathbb{F}}_q)$, which restricts to an automorphism of the finite group $\mathbf{G}(q^{ab})$. Here, we have $\alpha_b\colon \overline{\mathbb{F}}_q\to\overline{\mathbb{F}}_q$; $x\mapsto x^{q^b}$, which restricts to $\alpha$ on $\mathbb{F}_{q^a}$ (since we assume that $b\equiv 1$ modulo $a$). The map $\sigma$ is a diagram automorphism of order $a$. Extend $\alpha$ to $E\coloneqq E_{\mathcal{B}}$ by Lemma~\ref{lem:ext_fld_aut}, where $\mathcal{B}=\set{b\in\mathbb{Z}_+}[b\equiv 1\text{ mod }a]$, by setting $\alpha\coloneqq\bigcup_{b\in\mathcal{B}}{\rest{\alpha_b}_{\mathbb{F}_{q^{ab}}}}$. Write $\delta=\alpha\sigma$. Then 
$$
{}^a\mathbf{G}(q^{ab})=\set{g\in\mathbf{G}(q^{ab})}[g.\delta=g]=V_{\mathbb{F}_{q^{ab}}}(\mathcal{P}\cup\mathcal{P}_+),
$$
where $\mathcal{P}\subseteq\mathbb{F}_{q^a}[X_{ij}]$ is finite and does not depend on $b$ such that $V_{\mathbb{F}_{q^{ab}}}(\mathcal{P})=\mathbf{G}(q^{ab})$, and $\mathcal{P}_+\subseteq\mathbb{F}_{q^a}[X_{ij}^{\alpha^k}]$ are the $\gensubgrp{\alpha}$-twisted polynomial equations coming from the condition $g.\delta=g$. Hence we can take $\mathcal{Q}\coloneqq\mathcal{P}\cup\mathcal{P}_+$.
	
We define the ring homomorphism $\varphi\colon\overline{\mathbb{F}}_q[X_{ij}]\to\overline{\mathbb{F}}_q[Y_{ij}^k]$ by
$$
X_{ij}\mapsto\sum_{l=0}^{a-1}e_l Y_{ij}^l,
$$
for all $i,j$, where $e_0,\ldots,e_{a-1}$ form an $\mathbb{F}_q$-basis of $\mathbb{F}_{q^a}$. Extend $\varphi$ to all of $\overline{\mathbb{F}}_q[X_{ij}^{\alpha^k}]$ by setting $X_{ij}^{\alpha^k}\mapsto\sum_{l=0}^{a-1}(e_l.\alpha^k) Y_{ij}^l$ for all $i,j,k$. Let $b\in\mathbb{Z}_+$ with $b\equiv 1$ modulo $a$, and $\mathcal{R}\subseteq\mathbb{F}_{q^a}[X_{ij}^{\alpha^k}]$ be a finite set of twisted polynomial equations. Define 
$$
\mathcal{R}'\coloneqq\set{r_0,\ldots,r_{a-1}\in\mathbb{F}_q[Y_{ij}^k]}[\exists r\in\mathcal{R}.\varphi\colon r=\sum_{l=0}^{a-1}{e_lr_l}].
$$ 	
Then we have
\begin{equation}\label{eq:var_corresp}
V_{\mathbb{F}_{q^{ab}}}(\mathcal{R})\cong_\varphi V_{\mathbb{F}_{q^b}}(\mathcal{R}')\text{ and }
V_E(\mathcal{R})\cong_\varphi V_{E_\alpha}(\mathcal{R}'),
\end{equation}
so that for $\mathcal{R}=\mathcal{Q}$ we obtain
\begin{equation}\label{eq:vars}
{}^a\mathbf{G}(q^{ab})=V_{\mathbb{F}_{q^{ab}}}(\mathcal{Q})\cong_{\mathbb{F}_{q^b}} V_{\mathbb{F}_{q^b}}(\mathcal{Q}') \text{ and }
{}^a\mathbf{G}(E)=V_E(\mathcal{Q})\cong_{E_\alpha} V_{E_\alpha}(\mathcal{Q}'),
\end{equation}
where $\alpha$ is extended to $E$ by setting $\alpha\colon x\mapsto x^{q^b}$ on each finite subfield  $\mathbb{F}_{q^{ab}}$ (which exhaust $E$; as above).
	
Let us prove this briefly: We need to show that the map $\beta\colon\mathbb{F}_{q^b}^{\oplus a}\to\mathbb{F}_{q^{ab}}$ given by 
$$
(f_0,\ldots,f_{a-1})\mapsto\sum_{l=0}^{a-1}{e_lf_l}
$$ 
is a bijection. It is easily checked that $\im(\beta)$ is a $\mathbb{F}_{q^a}$- and $\mathbb{F}_{q ^b}$-vector subspace of the codomain of $\beta$. Hence it must have at least $q^{\lcm(a,b)}=q^{ab}$ elements and so $\beta$ is surjective.

\medskip

Assume that the polynomial $p\in\overline{\mathbb{F}}_q[X_{ij}]$ does not vanish on $\mathbf{G}(\overline{\mathbb{F}}_q)$, so that, by Zariski density (cf.\ Lemma~\ref{lem:twst_grp_zar_dense}), it does also not vanish on ${}^a\mathbf{G}(E)$. Then $p.\varphi$ does not vanish on $V_{E_\alpha}(\mathcal{Q}')\subseteq V_{\overline{\mathbb{F}}_q}(\mathcal{Q}')$.
Consider the intersection $G\cap V_{\mathbb{F}_{q^a}}(p)={}^a\mathbf{G}(q^a)\cap V_{\mathbb{F}_{q^a}}(p)$. This is a finite set which is in bijection with $V_{\mathbb{F}_q}(\mathcal{Q}',p.\varphi)$. The cardinality of this can be estimated by 
$$
\card{V_{\mathbb{F}_q}(\mathcal{Q}')\cap V(p.\varphi)}\leq\deg(V_{\overline{\mathbb{F}}_q}(\mathcal{Q}'))\deg(p.\varphi)q^{\dim(V_{\overline{\mathbb{F}}_q}(\mathcal{Q}'))-1},
$$
where $V_{\overline{\mathbb{F}}_q}(\mathcal{Q}')=\overline{V_{E_\alpha}(\mathcal{Q}')}$ (which holds by \cite{borel2012linear}*{Corollary~18.3}) is connected as the algebraic $E_\alpha$-group $V_{E_\alpha}(\mathcal{Q}')$ is generated by root subgroups (since all the finite subgroups $V_{\mathbb{F}^{q^b}}(\mathcal{Q}')\cong{}^a\mathbf{G}(q^{ab})$ have this property for $b\in\mathcal{B}$). 
		
Furthermore, $\deg(V_{\overline{\mathbb{F}}_q}(\mathcal{Q}'))=O_{\mathbf{G}}(1)$, $\deg(p.\varphi)=\deg(p)\leq l(n-1)$, and $\dim(V_{\overline{\mathbb{F}}_q}(\mathcal{Q}'))=\dim(\mathbf{G})$, since, by the order formulas
$\card{{}^a\mathbf{G}(q^{ab})}=
\card{V_{\mathbb{F}_{q^b}}(\mathcal{Q}')}=\Theta(q^{b\dim(\mathbf{G})})$
for all $b\in\mathbb{Z}_+$ with $b\equiv 1$ modulo $a$. Hence, we obtain 
\begin{align*}
\frac{\card{G\cap V_{\mathbb{F}_{q^a}}(p)}}{\card{G}}=\frac{O_{\mathbf{G}}(\deg(p) q^{\dim(\mathbf{G})-1})}{\Theta(q^{\dim(\mathbf{G})})}=O_{\mathbf{G}}(\deg(p)q^{-1}).
\end{align*}
The desired conclusion follows precisely as in the previous subsection, since $[G:G']$ is bounded, which yields Proposition~\ref{BreGreGurTaoProp}.

\section{Short identities for almost simple groups}\label{sec:almost_simple_grps}

We prove the following lemma, which shows that a finite almost simple group might have a mixed identity of bounded length, whereas its simple socle does not.

\begin{lemma}
The group $\PGO_n^\bullet(q)$ has a mixed identity of length eight when $p\neq 2$.
\end{lemma}

\begin{proof}
Let $t\coloneqq 1_V+h$ be an \emph{Eichler transformation}, where 
$$
x.h=f(x,u)v-f(x,v)u
$$ 
for two linearly independent isotropic vectors $u\perp v$) and 
$$
x.l\coloneqq f(x,w)w
$$ 
a rank-one operator for $w\in V$ anisotropic such that the reflection 
$$
i\coloneqq r_w=1_V-2\frac{1}{f(w,w)}l
$$ 
lies in $\GO_n^\bullet(q)$. Let $g\in\GO_n^\bullet(q)$ be an arbitrary element. Define 
$$
c\coloneqq i^gti^gt=i^g(1_V+h)i^g(1_V+h)=1_V+i^ghi^g+h+i^ghi^gh.
$$
Regarding the last term in the previous equation, we see that (respecting $h^2=0$)
\begin{align}
i^ghi^gh&=\left(1-\frac{2}{f(w,w)}l\right)^gh\left(1-\frac{2}{f(w,w)}l\right)^gh\nonumber\\
&=\frac{4}{f(w,w)^2}l^ghl^gh-\frac{2}{f(w,w)}hl^gh\label{eq:help6}
\end{align}
Moreover, we compute:
\begin{align*}
x.l^ghl^g &=f(f(x,w.g)(w.g).h,w.g)w.g\\
&=f(x,w.g)f((w.g).h,w.g)w.g=0,
\end{align*}
since for any vector $v\in V$ we have 
\begin{align*}
f(v.(1+h),v)&=f(v,v.(1-h))=f(v,v)+f(v.h,v)\\&
=f(v,v)-f(v.h,v)
\end{align*}
implying that $f(v.h,v)=0$ (as $p\neq 2$). Take $v=w.g$ to complete the argument.
Whence we obtain in Equation~\eqref{eq:help6}
$$
i^ghi^gh=-\frac{2}{f(w,w)}hl^g h.
$$	
But then the critical term $i^ghi^gh$ in the expression of $c$ is invariant under conjugation with $i^g$:
\begin{align*}
i^g(i^ghi^gh)i^g&=\left(1-\frac{2}{f(w,w)}l^g\right)\left(-\frac{2}{f(w,w)}hl^g h\right)\left(1-\frac{2}{f(w,w)}l^g\right)\\
&=-\frac{2}{f(w,w)}hl^g h,
\end{align*}
since $l^ghl^g=0$. Then $w(g)=[c,i^g]=1$, and $w$ is of length eight and it descends to a mixed identity of $\PGO_n^\bullet(q)$. Thus we are done.
\end{proof}

\begin{remark}
Indeed, for $n=2m$ even, the finite simple group $\POmega_{2m}^{\pm}(q)\trianglelefteq\PGO_{2m}^\pm(q)$ does not admit a mixed identity of bounded length (by Theorem~\ref{thm:main_finintro}), whereas its full automorphism group does have a such a mixed identity.    
\end{remark}

\section*{Acknowledgments}

We thank Vadim Alekseev and Ralf Köhl for illuminating discussions about the project.

\end{document}